%% file: siam-manuscript.tex
\documentclass[final]{siamltex}
\usepackage[utf8]{inputenc}
\usepackage{hyperref}
\usepackage{booktabs}
\usepackage[title]{appendix}
\usepackage{graphicx}
\graphicspath{{figs/}}
\usepackage[caption=false]{subfig}
\usepackage{xcolor}
\usepackage{natbib}
\setcitestyle{numbers,square}

\usepackage{common-symbols}
\newcommand{\email}[1]{\protect\href{mailto:#1}{#1}}

\renewcommand{\rev}[1]{#1}

\title{Solving Optimal Control Problems of Rigid-Body Dynamics with Collisions Using the Hybrid Minimum Principle}

\author{Wei Hu\thanks{Princeton University}
\and Jihao Long\thanks{Princeton University}
\and Yaohua Zang\thanks{Zhejiang University}
\and Weinan E\thanks{Peking University and Princeton University}
\and Jiequn Han\thanks{Flatiron Institute (\email{jiequnhan@gmail.com})}
}

\begin{document}

\maketitle

\begin{abstract} %
\input{hmp-0-abstract}
\end{abstract}


\input{hmp-1-introduction}
\input{hmp-2-problem-formulation}
\input{hmp-3-hmp}
\input{hmp-4-alg-0-summary}
\input{hmp-4-alg-2-details}

\input{hmp-5-exp}
\input{hmp-6-comparisons}
\input{hmp-6-future-works}

\bibliographystyle{plain}
\bibliography{ref}

\appendix
\input{hmp-7-appendix-algorithms}

\input{hmp-7-appendix-1-exp-details}
\input{hmp-7-appendix-multiple-collisions}

\end{document}

%% file: hmp-0-abstract.tex
Collisions are common in many dynamical systems with real applications.
They can be formulated as hybrid dynamical systems with discontinuities automatically triggered when states transverse certain manifolds.
We present an algorithm for the optimal control problem of such hybrid dynamical systems
based on solving the equations derived from the hybrid minimum principle (HMP).
The algorithm is an iterative scheme following the spirit of the method of successive approximations (MSA), and it is robust to undesired collisions observed in the initial guesses. 
We propose several techniques to address the additional numerical challenges introduced by the presence of discontinuities.
The algorithm is tested on disc collision problems whose optimal solutions exhibit one or multiple collisions.
Linear convergence in terms of iteration steps and asymptotic first-order accuracy in terms of time discretization are observed when the algorithm is implemented with the forward-Euler scheme.
The numerical results demonstrate that the proposed algorithm has better accuracy and convergence than direct methods based on gradient descent. 
Furthermore, the algorithm is also simpler, more accurate, and more stable than a deep reinforcement learning method.

%% file: hmp-1-introduction.tex
\section{Introduction}%
\label{sec:introduction}
Discontinuities are ubiquitous in dynamical systems governing many applications, \eg robotics.
They can happen either internally as an intrinsic property of the dynamical system or be triggered by external control signals.
These mostly smooth systems with the presence of discrete discontinuous phenomena are called hybrid systems.
They can be classified according to the different kinds of hybrid phenomena present in the system \citep{framework94, framework98}.

In this paper, we study the optimal control of dynamical systems with discrete discontinuities and a Bolza-type objective.
Our interest lies in those discontinuities that are jumps of the state variables
and are triggered automatically when the states transverse certain manifolds (jumping manifolds).
These discontinuities are called \emph{autonomous state jumps}.
Rigid body collisions are models for these discontinuities, which are very common in robotics.

Another type of hybrid dynamical system that is widely studied in the literature is called \emph{controlled hybrid system},
\ie when and how the discontinuities affect the system are directly controlled by extra control variables.
Such hybrid phenomena are found in many applications, such as
impulsive control of satellite rendezvous \citep{controlled_app_satellite}, Chua's circuit \citep{controlled_app_circuit}, and hybrid vehicle \citep{hybid_vehicle,controlledHSBook}.
Most strategies for solving the optimal control problem of the hybrid system decouple the continuous control and the discrete control, including the switching times, switching sequence, jump magnitudes, \etc.
For a comprehensive review that focuses on controlled hybrid dynamics, we refer the readers to \cite{controlledHSBook} and the references therein. 

However, in systems with autonomous state jumps, these hybrid phenomena are tightly linked to the continuous control; thus, it is difficult to decouple the continuous control and the discrete phenomena. 
Compared with controlled hybrid systems, autonomous state jumps introduce extra nonlinearity into the problem due to the undetermined times and states at jumps.
Furthermore,
the total cost may change dramatically when the updated control triggers a new collision or avoids an existing collision in the solution process. 
These are all challenging issues for optimization.

Dynamic programming and trajectory optimization are two basic approaches for solving optimal control problems. Dynamic programming provides a global optimal closed-loop control based on Bellman’s principle of optimality \citep{bellman2003rand}. It suffers from the curse of dimensionality and generally does not scale to high dimensional systems. Trajectory optimization solves for a locally optimal trajectory and usually provides an open-loop solution.
Some methods lie in between, including differential dynamical programming (DDP, \cite{DDP}), iterative linear-quadratic regulator (iLQR, \cite{iLQR}),  sequential linear-quadratic control (SLQ, \cite{SLQ}), and iterative linear-quadratic-Gaussian method (iLQG, \cite{iLQG}).
These methods use linear or quadratic approximations around the nominal trajectory for the dynamics or costs to find a local closed-loop control.
Since these methods require derivatives of the dynamics, some smoothing techniques that sacrifice accuracy are necessary for their applications in discontinuous systems~\citep{humanoid,ddp_quadrupeds}.

In this paper, we will tackle the discontinuous model directly and find the open-loop optimal control.
There are two main categories of approaches for the open-loop optimal control problems: direct and indirect methods.
Direct methods turn the infinite-dimensional control problems into finite-dimensional nonlinear programming problems through discretization; thus, they are often described as discretize-then-optimize.
Being hybrid introduces extra difficulties. For example, the control problem of a frictional system is turned into a mixed-integer programming problem \citep{mixed_integer}.
In \cite{discontinuous_system}, it is pointed out that the discretize-then-optimize approach may lead to incorrect gradients.
Recently, differentiable physics simulation~\citep{Hu2020DiffTaichi,zhong2023improving} has also been investigated for solving open-loop optimal control problems involving discontinuities brought by contacts. However, accurately computing the associated gradients remains a challenge that calls for further research efforts~\citep{zhong2022differentiable}.
In \cite{autoAlgGradient,gradient_modified}, gradient-type algorithms are proposed for the autonomous hybrid systems where the adjoint equations are used to compute the gradients (with respect to the controls).
In this approach, the discontinuity information is implicitly reflected in the gradients and thus influences the continuous controls when applying iterative gradient-based algorithms.
Indirect methods use necessary conditions for the optimality of the control problem; they are often described as optimize-then-discretize. 
They turn the optimal control problem into a system of differential-algebraic equations. The necessary conditions are formulated for the continuous-time problem; thus, they have less to do with how the dynamical system is discretized. 
Pontryagin's minimum principle (PMP) \citep{pmp2} is a well-known set of first-order necessary conditions in optimal control.
Based on the hybrid maximum principle, an extension of PMP to the hybrid dynamical system, \cite{autoSplit} proposed a numerical algorithm (HMP-MAS, short for hybrid maximum principle multiple autonomous switchings) that handles the discontinuities explicitly.
For a (estimated) given number of switches and an initial guess of switching times and states,
HMP-MAS alternates between computing the optimal controls inside each continuous interval and updating switching times and states by gradient descent.
The constraints that switching states must be on the switching manifolds are enforced through the penalty method.  
This method was proposed for the switching system but can be extended to the state-jump system; see \Cref{sub:comparison_to_hmp_mas}.

\subsection{Proposed method and contributions
}
This paper proposes an iterative method for the open-loop optimal control of dynamical systems with autonomous state jumps (collisions).
We begin with a hybrid minimum principle (HMP) which extends the PMP to the hybrid dynamical systems with autonomous state jumps.
Following HMP and the method of successive approximation (MSA, \cite{msa2}), we propose a modified version that incorporates a relaxation technique to improve convergence.

We implement the algorithm with the forward-Euler scheme and numerically address several difficulties introduced by the discontinuities of the hybrid system. 
We propose several disc collision problems with one or multiple collisions and derive their analytical optimal solutions.
\rev{The method is then tested on these problems, demonstrating stable linear convergence with respect to the number of iteration steps (see \Cref{sub:convergence_rate} for a detailed study). Furthermore, asymptotic first-order accuracy is observed for the error with respect to the time step size.}

\rev{
The contributions of this paper are threefold.
First, we propose an iterative method based on MSA with a helpful relaxation term and demonstrate its ability to solve optimal control problems for rigid bodies with collisions both accurately and efficiently.
Second, we address several numerical challenges introduced by collisions, including the misalignment of discontinuities in the control and state variables, as well as the inherent discontinuities in the control itself.
These issues are critical for ensuring the stability and efficiency of the algorithm.
Third, we introduce a typical rigid body problems with collisions and provide analytical solutions, which can serve as a baseline problem for physics simulations and control tasks \citep{zhong2023improving, zhong2022differentiable}.
}

\rev{
As demonstrated in \Cref{sec:comparison}, a comparison with the HMP-MAS method, a direct method, and deep reinforcement learning approaches highlights several advantages of the proposed method. 
Unlike the HMP-MAS method, the algorithm automatically avoids unnecessary collisions and does not require an accurate prediction of the number of collision times.
It generally converges more rapidly and produces more accurate solutions than gradient-based direct methods.
In addition, it outperforms reinforcement learning approaches in both efficiency and accuracy.
}

\subsection{Organization and notations}%
\label{sub:notations}

This article is organized as follows. In \Cref{sec:formulation_and_assumptions}, we formulate the dynamical systems.
The hybrid minimum principle is presented in  \Cref{sec:the_pontryagin_s_maximal_principle_and_algorithm}.
\Cref{sec:discretization} introduces an HMP-based numerical algorithm and takes into account several numerical issues.
We demonstrate the performance of the proposed algorithm for several examples in \Cref{sec:experiments}.
Comparisons to the HMP-MAS, a direct method, and a deep reinforcement learning method are then discussed in \Cref{sec:comparison}.
Finally, we conclude in \Cref{sec:discussions_and_future_work}.

We use $\vert\bz\vert$ to denote the Euclidean norm for any $\bz$ in the Euclidean space.
For a scalar function $h=h(\bz): \bR^n \rightarrow \bR$, its derivative with respect to $\bz$ is a row vector $\partial_z h = h_z = [h_{z_1}, \cdots, h_{z_n}] $.
For a vector-valued function $\bh=[h_1, \cdots, h_{k}]^T: \bR^n \rightarrow \bR^k$, its Jacobian is defined as 
\[
\bh_z = 
\begin{bmatrix}
\partial_{z_1} h_1 &\cdots &\partial_{z_n} h_1 \\
\partial_{z_1} h_2 &\cdots &\partial_{z_n} h_2 \\
\vdots &\ddots &\vdots \\
\partial_{z_1} h_k &\cdots &\partial_{z_n} h_k \\
\end{bmatrix} \in \mathbb{R}^{k\times n}.
\]
In two dimensions, we also use $x$ and $y$ as subscripts to denote the $x$ and $y$ coordinates, respectively.
We define the $L^p$ norm for any vector-valued functions $\bh\in L^p(a,b;\bR^k)$ as
\begin{equation}
    \label{eq:def_lp_norm}
    \norm[L^p(a,b)]{\bh} = (\int_{a}^{b} \norm{\bh(t)}^p~dt)^{\frac{1}{p}}.
\end{equation}
When $p=\infty$, $\norm[L^\infty(a, b)]{\bh} = \norm[L^\infty(a, b)]{\norm{\bh}}$. When $(a, b) = (0, T)$, we also write $\norm[p]{\bh}=\norm[L^p(0,T)]{\bh}$.

%% file: hmp-2-problem-formulation.tex
\section{Problem formulation}%
\label{sec:formulation_and_assumptions}
\rev{
The time interval considered in this work is always fixed as $[0,T]$. We denote the state of the system by $\bx(\cdot)\in\bR^m$ and
the control by $\bu(\cdot)\in U\subset\bR^{m_u}$. We assume that $U$ is a compact convex set.
} Given the initial condition $\bx(0) = \bx_0$ and the control process $\bu(t)$, the evolution of state $\bx(t)$ is governed by piecewise smooth dynamics:
\begin{equation}
    \label{eq:forward_state}
    \dot{\bx}(t) = \bff(\bx(t), \bu(t)), t\in(\gamma_i, \gamma_{i+1}), i=0, \cdots, \NJp, 
\end{equation}
where $0=\gamma_0<\gamma_1<\cdots<\gamma_{\NJp}<\gamma_{\NJp+1}=T$.
For each $i= 1, \cdots, \NJp$, a jump function $g$ is applied to the state at time $\gamma_i$,
\begin{equation}
    \label{eq:forward_jump}
    \bx(\gamma_i^+) = \bfg(\bx(\gamma_i^-)).
\end{equation}
Here $\bx(\gamma_i^\pm)$ denote the value of $\bx$ at $\gamma_i$ evaluated from the right/left time interval, \ie, $\lim_{t\in (\gamma_i, \gamma_{i+1}), t\rightarrow \gamma_i}\bx(t)$ / $\lim_{t\in (\gamma_{i-1}, \gamma_{i}), t\rightarrow \gamma_i}\bx(t)$, respectively.

Note that the number of jumps $\NJp$ and the jump time $\gamma_i$ are not prescribed;  instead, they are functions of the control $\bu$.
They are determined by a ``collision detection'' function $\psi(\bx)\ge0$ that checks whether a jump of state occurs.
Namely,
\begin{equation}
    \label{eq:forward_collision}
    \psi(\bx(\gamma_i^-)) = 0,
\end{equation}
and  $\psi(\bx(\cdot)) > 0$ at $\gamma_i^+$ and all $t$ other than $\gamma_i^-$.
That is why we call such jumps as autonomous state jumps. We assume both $\psi$ and $\bfg$ are smooth in an open set containing $\{\bx: \psi(\bx) = 0\}$.

The optimal control problem is then formulated as
\begin{mini!}|l| {\rev{\bu(\cdot)\in U}}{J(\bu(\cdot))=\phi(\bx(T)) + \sum_{i=0}^{\NJp}\int_{\gamma_{i}}^{\gamma_{i+1}}L\,dt,} 
    {\label{eq:opt_problem}}{}  
    \addConstraint{\text{\cref{eq:forward_state,eq:forward_jump,eq:forward_collision}}.}
\end{mini!}
Here, $\phi$ is the terminal cost function, and $L=L(\bx(t),\bu(t))$ is the running cost function; both are smooth.
We assume that $\bff, \phi, L$ are time-independent and that $\bff, L$ do not switch across different intervals $(\gamma_i, \gamma_{i+1})$ in this paper for the simplicity of notations. 
The considered minimum principle and proposed numerical algorithm can work for time-dependent problems and switching systems straightforwardly.
See also \eg \cite{hmp_pak1} for theory for systems with autonomous jumps and switchings. \rev{For the wellposedness of the hybrid optimal control problem, we refer to \cite{goebel2019existence,ferrante2019certifying,altin2023regularity}.}

To ensure that the above model is well-posed, we shall require that the trajectory of the state $\bx$ transverses the manifold $\{\psi=0\}$ where jumps of the state take place. That is, the state $\bx$ should not move tangentially on the manifold, or 
\begin{equation}
\label{eq:transversality}
0\ne \frac{d}{dt}\psi(\bx(t)) = \psi_{x}(\bx(t))\dot{\bx}(t) = [\psi_{x} \bff]\vert_t.
\end{equation}
Here $\psi_{x}$ is the normal vector for the manifold and $\dot{\bx}=\bff$ is the instantaneous \rev{velocity} of the state $\bx$.
In the context of collisional rigid body dynamics, $\psi_{x} \bff$ is the relative velocity projected to the contact normal vector, which should be nonzero for the collision to take place. See Section \ref{sec:experiments} for a concrete example.

%% file: hmp-3-hmp.tex
\section{Hybrid minimum principle}%
\label{sec:the_pontryagin_s_maximal_principle_and_algorithm}
Pontryagin's maximum principle for standard optimal control can be extended to the discontinuous dynamics introduced here.
We shall call this extension the hybrid minimum principle (HMP).
To begin with, we introduce the {\it Hamiltonian} $H:\bR^m\times \bR^m\times U\rightarrow\bR$
\begin{equation}
    \label{eq:hamiltonian}
    H(\bx, \blambda, \bu) = \blambda^T \bff(\bx,\bu) + L(\bx,\bu).
\end{equation}
\begin{theorem}[Hybrid Minimum Principle]
    Let the state-control pair $\{\bx(t), \bu(t)\}$ be the optimal solution associated with \eqref{eq:opt_problem},
    then there exists a costate process $\blambda(t): [0,T]\rightarrow\bR^m$ that satisfies
    \begin{equation}
        \label{eq:backward_costate}
        \dot{\blambda}(t) = -H_x^T(\bx(t),\blambda(t), \bu(t)), \quad t\in(\gamma_i, \gamma_{i+1}),\; 
    \end{equation}
    for $i=0, \cdots, \NJp$, with the terminal condition
    \begin{equation}
        \label{eq:backward_terminal}
        \blambda(T) = \phi^T_{x}(\bx(T)),
    \end{equation}
    and backward jump conditions, $i=1, \cdots, \NJp,$
    \begin{equation}
        \label{eq:backward_jump}
        \blambda^T(\gamma_i^-) = \blambda^T(\gamma_i^+)\bfg_{x}(x(\gamma_i^-)) + \eta_i\psi_{x}(\bx(\gamma_i^-)),
    \end{equation}
    where $\eta_i \in \bR$ satisfies the following equation (variation of collision time)
    \begin{equation}
        \label{eq:backward_eta}
        \begin{split}
            \blambda^T(\gamma_i^+) [\bfg_{x}(\bx(\gamma_i^-))\bff\vert_{\gamma_i^-} -\bff\vert_{\gamma_i^+}] &+ \eta_i
        \psi_{x}(\bx(\gamma_i^-))\bff\vert_{\gamma_i^-}\\
        &+L\vert_{\gamma_i^-} - L\vert_{\gamma_i^+} = 0.
        \end{split}
    \end{equation}
    Here, $h\vert_{s^{\pm}}$ denotes function $h$ evaluated at $s$ as limits from right and left, respectively.
    In addition, the optimal control $\bu$ minimizes the Hamiltonian for each $t\in[0, T]$:
    \begin{equation}
        \label{eq:minimal_condition}
        \bu(t) = \arg\min_{\bv\in U} H(\bx(t), \blambda(t), \bv).
    \end{equation}
    \label{thm:hmp}
\end{theorem}
We refer to \cite{volin1969maximum} for a proof of this theorem.
A few remarks on PMP and HMP
are in order.
\begin{remark}
\label{rem:transversality}
By the assumption of transversality \cref{eq:transversality}, we have $\psi_{x} \bff\ne 0$ when $\psi=0$. Therefore the equation \cref{eq:backward_eta} is always non-degenerate.
We note that the condition \cref{eq:backward_eta} from the variation of collision time is equivalent to the Hamiltonian constancy condition quoted in other literature, i.e. $H\vert_{\gamma_i^-}=H\vert_{\gamma_i^+}$. 
\end{remark}
\begin{remark}
We have omitted the abnormal multiplier in the statement of Theorem \ref{thm:hmp} for brevity.
In a more complete form, there is a technicality involving multiplier $\lambda_0\ge0$ in the Hamiltonian \eqref{eq:hamiltonian} $H=\blambda^T\bff + \lambda_0L$ and in the terminal condition \eqref{eq:backward_terminal} $\blambda(T) = \lambda_0\phi^T_{x}(\bx(T))$.
Besides, the $\lambda_0$ also presents in front of  $L\vert_{\gamma^-_i}, L\vert_{\gamma^+_i}$ in \cref{eq:backward_eta}, which can be seen from its equivalence to $H\vert_{\gamma_i^-}=H\vert_{\gamma_i^+}$.
In the case where $\lambda_0 =0$ is the only candidate, the problem is singular and, in some sense, ill-posed \citep{optimal_control_book}.
Otherwise, we can rescale the costate $\blambda$ so that we can always let $\lambda_0=1$ (note that the additional $\eta_i$ in \cref{eq:backward_jump,eq:backward_eta} does not prevent $\blambda$ from rescaling).
We shall only consider this case and take $\lambda_0=1$ in this paper.
\end{remark}

\begin{remark}
\label{rem:hu_implicit}
Most proofs of PMP and HMP resort to the calculus of variations with some special classes of control variations. 
We note that, 
$H_{u}$ can be interpreted as the Fr\'echet derivative of the reduced total cost $\hat{J}$ with respect to control function $\bu$ in $L^2$, where
\[
\hat{J}(\bu(\cdot)) = \phi\left(\bx\left(T; \bu(\cdot)\right)\right) + \sum_{i=0}^{\NJp}\int_{\gamma_i}^{\gamma_{i+1}}L(\bx(t; \bu(\cdot)), \bu(t))\,dt,
\]
and $\bx(t; \bu(\cdot))$ denotes the trajectory obtained with control $\bu(\cdot)$ according to \cref{eq:forward_state,eq:forward_jump,eq:forward_collision}.
This argument and its proof is often a step in the complete proof of the PMP and HMP.
For a formal derivation, see Chapter 6.2 of \cite{optimal_control_book3}. For a rigorous proof, see the proof of necessary conditions of optimality in Chapter 1.4 of \cite{optimal_control_book2} on continuous dynamics and
see those proofs of the HMP (\eg \cite{volin1969maximum}) for the present case.
\end{remark}

\begin{remark}
The PMP/HMP is related to the Karush-Kuhn-Tucker (KKT) conditions for constrained optimization problems.
In fact, the dynamic equations of $\bx$ (\cref{eq:forward_state,eq:forward_jump,eq:forward_collision}) can be viewed as an infinite-dimensional constraint. The costate $\blambda$ plays the role of a continuous-time analogy of the Lagrange multiplier.
The HMP/PMP is stronger than the KKT conditions in the sense that it asserts that $H$ is not only stationary but also minimized at the optimal control (\cref{eq:minimal_condition}).
Like KKT, the PMP/HMP is only a set of necessary conditions.
However, these sets of necessary conditions are often sufficient to provide a good (even optimal) solution given a reasonable initial guess. 
\end{remark}

\begin{remark}
HMP has been studied extensively in the literature.
In \cite{hmp_suss,hmp_hnp}, HMP has been formulated and proved for a general framework of the hybrid dynamic systems.
The HMP is proved by reducing the optimal control problem of the hybrid system to the canonical Pontryagin type in \cite{hmp_as_pmp}. 
Adding pathwise inequality constraints to the control problem will append related complementary equations to the resulting HMP \citep{hmp_pathwise_constraints}.
In \cite{hmp_impulsive_controlled_magnitude}, the HMP has been derived for autonomous hybrid dynamical systems whose jump magnitudes are additional control variables. The HMP has also been extended to manifolds where the costate is a trajectory on the cotangent bundle \citep{hmp_manifold}.   
\end{remark}

%% file: hmp-4-alg-0-summary.tex
\section{Algorithm}%
\label{sec:discretization}
In this section, based on the HMP in \Cref{thm:hmp} and the MSA algorithm, 
we develop an improved version for better numerical convergence.

Prior to the development of numerical algorithms, the jump function of the costate requires extra clarification.
Equations \eqref{eq:backward_jump} and \eqref{eq:backward_eta} are condensed into a jump function $G$ for simplicity as follows:
\begin{equation}
\label{eq:backward_jump_form1}
\blambda(\gamma_i^-) = \bfG(\blambda(\gamma_i^+),\bx(\gamma_i^+), \bx(\gamma_i^-), \bu(\gamma_i^+), \bu(\gamma_i^-)).
\end{equation}
Note that $G$ depends on $\bx(\gamma_i^+)$, $\bu(\gamma_i^+)$, and $\bu(\gamma_i^-)$, due to $\eta$.
Here $\bu(\gamma_i^{\pm})$ refers to the after/before-jump values of $\bu$.
To clarify the semantic meaning of the after/before-jump values, we introduce $\beta_i$, which denotes the $i$-th discontinuous point of $\bu$.
For the optimal solution, we have $\gamma_i = \beta_i$.
However, when $\bu$ approaches optimality, it is crucial to note that the locations of the discontinuities can be different, \ie, $\gamma_i\ne\beta_i$, even though the number of $\beta_i$ equals that of $\gamma_i$.
It implies that $\bu(\gamma_i^{\pm})$ is not a stable quantity; it might change discontinuously when $\gamma_i$ leaps over $\beta_i$.
We shall use $\bu(\beta_i^\pm)$ instead in designing numerical algorithms. This is also reflected later in \cref{eq:control_resolve_steepest} for approximating $\beta_i$.
Therefore, we rewrite equation \eqref{eq:backward_jump_form1} as
\begin{equation}
    \label{eq:backward_jump_form2}
    \blambda(\gamma_i^-) = \bfG(\blambda(\gamma_i^+),\bx(\gamma_i^+), \bx(\gamma_i^-), \bu(\beta_i^+), \bu(\beta_i^-)).
\end{equation}
Next, we will review the MSA algorithm and develop a modified version using the formula \eqref{eq:backward_jump_form2} instead.

At a high level, given a control, the forward equations \eqref{eq:forward_state} to  \eqref{eq:forward_collision} for the state $\bx(\cdot)$
and the backward equations 
\eqref{eq:backward_costate}, \eqref{eq:backward_terminal} and \eqref{eq:backward_jump_form2} for the costate $\blambda(\cdot)$
are satisfied by solving the corresponding forward and backward problems, respectively.
It remains to find the optimal control $\bu(\cdot)$ such that \eqref{eq:minimal_condition} is also satisfied.
To this end, we propose an iterative method. In the $k$-th iteration, given the control $\bu^k(\cdot)$, we compute $\bx^k(\cdot), \blambda^k(\cdot)$ by solving the forward and backward dynamics, respectively.
We then update the control using the information from minimizing the Hamiltonian \eqref{eq:minimal_condition}.
To be specific, we define the minimizer $\bu_{*}^{k+1}(\cdot)$ in the $k$-th iteration by
\begin{equation}
	\label{eq:pmp_max_u}
	\bu_{*}^{k+1}(t) = \arg\min_{\bv\in U } H(\bx^k(t), \blambda^k(t), \bv), \quad \forall t\in[0,T].
\end{equation}
If we update $\bu^{k+1} = \bu^{k+1}_{*}$, we arrive at the method of successive approximations (MSA)~\citep{msa2}.
To enhance the applicability and robustness of the algorithm in numerical computation, we introduce a relaxed version to update $\bu^{k+1}$ (see \eqref{eq:u_relax_update}) in Section~\ref{sec:relaxed_MSA}.

%% file: hmp-4-alg-2-details.tex
\subsection{The relaxed MSA algorithm}
\label{sec:relaxed_MSA}
Even for optimal control problems without discontinuities, the MSA method only converges for a restricted class of linear systems and diverges for simple linear quadratic optimal control problems~\citep{aleksandrov1968accumulation}.
One reason is that the control $\bu_{*}^{k+1}$ may stray too far from the region of validity approximated at $\bu^k$, and the total objective may even increase (see~\cite{msa3} for a more detailed explanation).
In \cite{msa3}, an extended version of the method of successive approximations is proposed to penalize the deviations from the dynamic constraints on the state and costate, \ie, $\dot{\bx}^k - H_{\lambda}(\bx^k,\blambda^k,\bu^{k+1})$ and $\dot{\blambda}^k+H_{x}(\bx^k,\blambda^k,\bu^{k+1})$.
In this paper, we propose to adopt the following update scheme with relaxation:
\begin{equation}
	\bu^{k+1}(t) = (1-\alpha) \bu^{k}(t) + \alpha \bu_{*}^{k+1}(t), \quad \forall t\in[0,T]
	\label{eq:u_relax_update}
\end{equation}
where $\alpha\in(0,1]$ is a relaxation parameter.
We shall use a small $\alpha$ in \cref{eq:u_relax_update} to restrict the update of the control. \rev{Note that since both $\bu^{k}(t)$ and $\bu_{*}^{k+1}(t)$ are in the compact convex set $U$, the update $\bu^{k+1}(t)$ also remains inside the admissible control set $U$ as a convex combination.}

We now discuss the numerical implementation of the iterative scheme presented above.
Each iteration consists of the following four steps:
\begin{enumerate}
	\item
	    Using the control from the last iteration, we numerically simulate the \emph{forward dynamics} \cref{eq:forward_state,eq:forward_jump,eq:forward_collision} utilizing the forward-Euler scheme with the collision times prediction.

	\item
      Solve the \emph{backward dynamics} \cref{eq:backward_costate,eq:backward_terminal,eq:backward_jump,eq:backward_eta} using the forward-Euler scheme (backward in time) and collision information from the forward simulation.
	\item
    \emph{Update the control} pointwise in time according to \cref{eq:pmp_max_u,eq:u_relax_update} with the newly computed state and costate.
	\item \rev{\emph{Test the convergence} by evaluating the $L^2$ norm of the Hamiltonian \cref{eq:hamiltonian}, and return the current solution if the stopping criterion is satisfied.}
\end{enumerate}
In the remaining of this section, we share details on the numerical discretization of each of these steps, particularly emphasizing challenges brought by discontinuities.
We begin by discretizing the time interval $[0, T]$ into $N$ uniform sub-intervals: $t_n = n\Delta t,$ $n=0,\cdots, N,$ $\Delta t = \frac{T}{N}$.
For simplicity, we omit the iteration superscript $k$ in our subsequent discussion of the discretization.
$\bx_n, \blambda_n, \bu_n$ denote the numerical approximations of the state, costate, and control at $t_n$, respectively.
A summary of the entire algorithm is in \Cref{alg:PMP}.

\begin{algorithm}[ht!]
	\caption{The relaxed MSA algorithm for solving optimal control with hybrid dynamics}
	\label{alg:PMP}
	\begin{algorithmic}[0]
		\State
		\textbf{Input:}
		the number of discretization time intervals $N$, the initial guess of the control $\bu^0 = [\bu_0^0, \bu_1^0, \cdots, \bu_{N-1}^0]$, \rev{the error tolerance $\delta>0$}.
		\State
		$k\gets 0$
		\While{True}
		\Program{Forward}
		\State \multiline{
		Plug $\bu^k$ into \cref{eq:fd_collision_time,eq:fd_case1,eq:fd_case2} to get $\bx^k=[\bx_0^k, \bx_1^k, \cdots, \bx_N^k]$,
		the active collision indexes set $\cC$ defined in \cref{eq:fd_active_collision_index},
		and the states $\bx_{c+\frac{1}{2}}^{\pm}$ around collision for $c\in\cC$.}
		\EndProgram
		\Program{Backward}
		\State \multiline{
		Using computed $\bx^k$ and derived information around collisions,
		obtain $\blambda^k=[\blambda_1^k, \cdots, \blambda_N^k]$
		by solving the backward dynamics with jumps described in \cref{eq:bw_dis_no_jump,eq:bw_dis_jump}.
		The control values $\bu_{c+\frac{1}{2}}^{\pm}$ in \cref{eq:bw_dis_jump} are computed by combining \cref{eq:control_resolve_steepest,eq:control_resolve_stable}.
		}
		\EndProgram
		\Program{Control Update}
		\State \multiline{
		Get control $\bu^{k+1}=[\bu^{k+1}_0, \cdots, \bu^{k+1}_{N-1}]$
		according to \cref{eq:pmp_max_u_dis,eq:u_relax_update_dis}.}
		\EndProgram

		\Program{Convergence Test}
		\State \multiline{
			Compute discrete $L^2$ norm of $H_{u}$ according to \cref{eq:dis_convergence_l2_norm}.
			If it is less than $\delta$, return the control $\bu^k$.
		}
		\EndProgram
		\State
		$k\gets k+1$ \;
		\EndWhile

		\State \textbf{Output:} the optimal control $\bu^k$
	\end{algorithmic}
\end{algorithm}

\paragraph*{Forward dynamics}
The numerical simulation of a hybrid dynamical system with discontinuities is not a simple task, and there are many related works in the literature \citep{forward_int_event_est1,forward_int_event_est2, forward_int_event_est3, forward_int_time_step1, forward_int_relax,forward_int_analysis1,forward_int_analysis2}.
The key idea involves an accurate estimation of the collision time using a root-finding scheme and a proper integration step around the collision by a variable-step integrator.
Our integrator estimates the collision time based on a linear approximation and adds an intermediate time step at the estimated time.
To be specific, at each step, we estimate the collision time by solving the equation of $s_n=s(\bx_n, \bu_n)$:
\begin{equation}
	s_n=\inf \{s > 0: \psi(\bx_n+s\bff(\bx_n, \bu_n)) = 0\}.
\label{eq:fd_collision_time}
\end{equation}
Here, $s_n=\infty$ if no collision is predicted from the current state $x_n$ with control $u_n$.
The forward dynamics is simulated as follows. For $n=0, \cdots, N-1$, if $s_n = s(\bx_n, \bu_n) < \Delta t$, an extra integration step at $t_n+s_n$ is inserted and the jump function $\bfg$ is applied:
\begin{equation}
    \label{eq:fd_case1}
    \begin{cases}
        \bx_{n+\frac{1}{2}}^- &= \bx_n + s_n \bff(\bx_n, \bu_n), \\
            \bx_{n+\frac{1}{2}}^+ &= \bfg(\bx_{n+\frac{1}{2}}^-), \\
            \bx_{n+1} &= \bx_{n+\frac{1}{2}}^+ + (\Delta t - s_n) \bff(\bx_{n+\frac{1}{2}}^+, \bu_{n+\frac{1}{2}}^+);
   \end{cases}  
\end{equation}
if $s_n = s(\bx_n, \bu_n) \ge \Delta t$,
    \begin{equation}
        \label{eq:fd_case2}
        \bx_{n+1} = \bx_n + \Delta t \bff(\bx_n, \bu_n).
\end{equation}
Here, $\bu_{n+\frac{1}{2}}^+= \bu_{n+1}$ if we only have a discrete sequence of the control.
In practice, estimating the collision time adds minimal complexity to the simulation. For most integration steps, where a collision does not occur within a given time step, a rapid and approximate calculation of $s_n$ is adequate.
We note that the estimation of the collision time does not add much complexity to the simulation in practice. At most integration steps $t_n$ where no collision happens within a time step, a fast but rough estimation of $s_n$ is sufficient.
Along with the forward simulation, we collect the active collision index set $\cC$ as
\begin{equation}
	\cC = \{n: s_n<\Delta t\}.
\label{eq:fd_active_collision_index}
\end{equation}

\paragraph*{Backward dynamics}
The discretization of the backward costate equations (\cref{eq:backward_costate,eq:backward_terminal,eq:backward_jump_form2}) is as follows.
With the terminal condition $\blambda_{N} = \phi_{x}^T(\bx_N, t_N),$ for $n = N-1, N-2, \cdots, 1$, we have
\begin{equation}
	\blambda_{n} = \blambda_{n+1} + \Delta t H_{x}^T(\bx_{n}, \blambda_{n+1}, \bu_{n}), \quad \text{ if } n\notin \cC,
	\label{eq:bw_dis_no_jump}
\end{equation}
and if $n\in\cC$,
\begin{subequations}
	\begin{align}
		\blambda_{n+\frac{1}{2}}^+ & = \blambda_{n+1} + (\Delta t - s_{n}) H_{x}^T(\bx_{n+\frac{1}{2}}^+, \blambda_{n+1}, \bu_{n+\frac{1}{2}}^+), \label{eq:bw_dis_jump_p1}                                                                               \\
  \blambda_{n+\frac{1}{2}}^- &= \bfG(\blambda_{n+\frac{1}{2}}^+,\bx_{n+\frac{1}{2}}^+, \bx_{n+\frac{1}{2}}^-, \bu_{n+\frac{1}{2}}^+, \bu_{n+\frac{1}{2}}^-)),
  \label{eq:bw_dis_jump_jump} \\
		\blambda_{n}               & = \blambda_{n+\frac{1}{2}}^- + s_{n}H_{x}^T(\bx_{n}, \blambda_{n+\frac{1}{2}}^-, \bu_{n}) \label{eq:bw_dis_jump_p2}.
	\end{align}
	\label{eq:bw_dis_jump}%
\end{subequations}
Here, the function jump function $G$ is defined in \cref{eq:backward_jump_form2}.
In practice, the discontinuities will introduce numerical instability, especially when solving \cref{eq:bw_dis_jump_jump}.
The absence of $\Delta t$ in \cref{eq:bw_dis_jump_jump} means that for $c\in\cC$, an error of $O(1)$ in $\bu^{\pm}_{c+\frac{1}{2}}$ will result in an error of $O(1)$ in $\blambda_c$, and thus in all $\blambda_k$ for $k\le c$.
Therefore, some extra numerical treatments are needed in order to ensure numerical stability and convergence speed when estimating $\bu_{c+\frac{1}{2}}^{\pm}$ in evaluating \cref{eq:bw_dis_jump}.
The following two paragraphs are devoted to a better estimation of $\bu_{c+\frac{1}{2}}^{\pm}$.

Ideally, the discontinuities of the optimal control and collisions should occur simultaneously.
However, as addressed in replacing \cref{eq:backward_jump_form1} by \cref{eq:backward_jump_form2}, in the algorithm with time discretization, they can misalign before convergence. During iterations, we observe significant delays in the discontinuities of the control in response to the shifting discontinuities in the state $\bx$ due to the relaxed updating rule of $\bu$ in \cref{eq:u_relax_update}.
By \cref{eq:backward_jump_form2},
instead of using the control values $\bu_{c+1}$ and $\bu_{c}$ that are closest to the collision (which can be viewed as a numerical approximation of $\bu(\gamma_i^{\pm})$), we propose to use the steepest varying control values near the collision (which can be viewed as a numerical approximation of $\bu(\beta_i^{\pm})$).
To be specific, we find
\begin{equation}
	c^* = \arg\max_{\vert n -c\vert<l} \vert\bu_{n}-\bu_{n+1}\vert,
	\label{eq:control_resolve_steepest}
\end{equation}
where $l$ is a given integer. Then $t_{c^*}$ can be understood as the numerical resolution of the discontinuous point $\beta_i$ of $\bu$.
The information around $t_{c^*}$ will be used to determine $\bu_{c+\frac{1}{2}}^+, \bu_{c+\frac{1}{2}}^-$, as explained in the next paragraph.
In practice, we choose $l= \lceil \tau N \rceil,  \tau\in(0, 1)$ so that the range in which we search for the discontinuity is independent of the discretization.
Namely, we resolve better control estimations locally in $(t_c-\tau T, t_c+\tau T)$, where $t_c=c\Delta t$.
\rev{
Any value of $\tau$ such that $\tau T$ is less than the shortest time between adjacent collisions and larger than the misalignment $\gamma_i - \beta_i$ will be suitable.
}
$\tau=0.05$ is chosen in this work.

In our experiments, critical oscillations were observed in the solution near collision points, primarily due to the discontinuous nature of the control in these regions.
To mitigate this, we propose a resolution strategy that extrapolates using control values located further from the collision points. A comprehensive analysis of the oscillation and the approach is detailed in \Cref{app:stable_estimate_near_collision}.
In this paper, we utilize a straightforward first-order extrapolation:
\begin{equation}
	\label{eq:control_resolve_stable}
	\bu_{c+\frac{1}{2}}^+ = \bu_{c^*+2}, \quad \bu_{c+\frac{1}{2}}^- = \bu_{c^*-1}.
\end{equation}
\paragraph*{Updating the control}
It remains to discretize the equation \eqref{eq:pmp_max_u}.
From the discretization of $\bx, \blambda$ described above, we obtain $\bx_0^k, \cdots, \bx_N^k$ and $\blambda_1^k, \cdots, \blambda_N^k$ in the $k$-th step by using $\bu^k$, and we need to update the control at $t_0, \cdots, t_{N-1}$.
Note that the control at $t_{N}=T$ is not used in the discretized forward and backward dynamics; thus, we do not need to update it.
For each $\bu_n$ to be updated, the costate information should be backpropagated from $\blambda_{n+1}$.
So we apply the following rule at each time step $n$,
\begin{equation}
	\label{eq:pmp_max_u_dis}
	\bu_{*, n}^{k+1} = \arg\min_{\bv\in U} H(\bx_n^k, \blambda_{n+1}^k, \bv),\quad n=0, \cdots, N-1.
\end{equation}
We then discretize \cref{eq:u_relax_update} pointwisely:
\begin{equation}
	\bu^{k+1}_n = (1-\alpha) \bu^{k}_n + \alpha \bu_{*, n}^{k+1},\quad n=0,\cdots,N-1.
	\label{eq:u_relax_update_dis}
\end{equation}
\Cref{eq:pmp_max_u_dis} defines a finite-dimensional optimization problem for each $n$, which can be solved in parallel.
In cases like the examples considered in this paper, there are analytical expressions for the minimizers of the Hamiltonian.
If $H$ is convex in $\bu$ for fixed $\bx,\blambda$, many efficient convex optimization algorithms can be used.
Otherwise, a general numerical optimizer can also perform well with the initial guesses from the last iterations and adjacent grids.

It is worth mentioning that the selection of indices in \cref{eq:pmp_max_u_dis} is consistent with the discretize-then-optimize approach in terms of domains of dependence. Specifically, the control variable $\bu$ at index $n$ is dependent on the state variable $\bx$ at index $n$ and the costate variable $\blambda$ at index $n+1$. To demonstrate, consider the following discrete problem without any collision:
\begin{mini!}|l| {\bu_0, \cdots, \bu_{N-1}}{\phi(\bx_N) + \sum_{i=0}^{N-1}L(\bx_i, \bu_i) \Delta t,}
{\label{eq:opt_problem_dis}}{}
\addConstraint{\bx_{n+1} = \bx_n + \Delta t \bff(\bx_n, \bu_n).}
\end{mini!}
Introducing a Lagrange multiplier $\blambda_n$ for the dynamic constraint at each time step and then deriving the first-order necessary condition yields
\begin{equation}
	\label{eq:discrete_opt_condition}
	H_{u}(\bx_{*,n}, \blambda_{*,n+1}, \bu_{*,n})=0,
\end{equation}
where the subscript $*$ denotes the optimal solution.
We can see that the time indexes in \cref{eq:discrete_opt_condition} are consistent with those in \cref{eq:pmp_max_u_dis}.

\paragraph*{Convergence test}
Since the forward dynamics of $\bx$ and the backward dynamics of $\blambda$ are always solved accurately in the previous steps, we only need to check the Hamiltonian minimization condition \cref{eq:minimal_condition} to determine if convergence is achieved.
The $L^2$ norm of $H_{u}$ (\cref{eq:minimal_condition}) serves as a natural convergence indicator.
As pointed out in Remark \ref{rem:hu_implicit}, $H_{u}$ also has the meaning of the derivative of the total cost with respect to the control $\bu$.
Given $\bu^k=[\bu^k_0, \cdots, \bu^k_{N-1}]$, $H_{u}$ is discretized as
\begin{equation}
	\label{eq:hu_dis}
	H_{u,n}^k = H_{u}(\bx^k_n, \blambda^k_{n+1}, \bu^{k}_n), \; n=0,\cdots, N-1.
\end{equation}
The indexes relation follows the same logic as in \cref{eq:pmp_max_u_dis}.
We then compute the numerical $L^2$ norm as
\begin{equation}
	\label{eq:dis_convergence_l2_norm}
	\sqrt{\sum_{n=0}^{N-1} (H_{u,n}^k)^2\Delta t}.
\end{equation}
We stop the algorithm when the above quantity is less than a pre-specified tolerance number $\delta$.

%% file: hmp-5-exp.tex
\section{Experiments}%
\label{sec:experiments}
This paper mainly focuses on those hybrid dynamics whose difficulties are introduced by collisions.
We will demonstrate the performance of Algorithm \ref{alg:PMP} in this scenario by applying it to a system of two-dimensional, homogeneous discs which characterized solely by their masses and radii, in frictionless rigid body dynamics.
We remove the degrees of freedom associated with the orientations and angular velocities as these quantities are constant in this context.

There are always two discs in our experiments.
The state $\bx$ is composed of the positions and velocities of the two discs, \ie $\bx\in\bR^8.$ 
The objective is to maneuver disc 1 by applying a force that causes it to collide with disc 2, aiming to position disc 2 near predetermined locations at terminal time $T$.
Due to the translation invariance, the prescribed terminal location is always set to the origin $(0, 0)$.

Let us write $\bx = [\bx_{q_1}^T, \bx_{q_2}^T, \bx_{v_1}^T, \bx_{v_2}^T]^T$, where $\bx_{q_i}, \bx_{v_i}$ are the position and the velocity of disc $i$, respectively. Similarly, the costate is written as $\blambda = [\blambda_{q_1}^T, \blambda_{q_2}^T, \blambda_{v_1}^T, \blambda_{v_2}^T]^T$. The control objectives include a terminal cost $\phi$, quantifying the distance between disc 2 and the terminal location $(0,0)$ as $\phi(\bx) = \vert\bx_{q_2}\vert^2 $,
and a running cost $L$, expressing the square of the magnitude of the applied force,
\begin{equation}
    \label{eq:eg_running_cost}
    L(\bx, \bu) = \epsilon \vert\bu\vert^2,
\end{equation}
where $\epsilon>0$ is a fixed parameter.
The admissible control set $U$ is $\bR^2$.
We also explore non-convex running costs (see \Cref{sub:L1-cost}).
Detailed descriptions of the dynamics, costate dynamics, collision detection, and jump functions for the two-disc system are in \Cref{app:details_of_exp}.

Like most iterative algorithms in nonlinear optimal control, a suitable initial guess is necessary to avoid undesired local minima.
This is more demanding for the discontinuous systems considered in this paper.
For instance, in our disc collision problems, the terminal cost only depends on disc 2, and the control only accelerates disc 1. If the initial control does not trigger a state jump of disc 2, any local optimization algorithm is likely to converge to zero control ($\bu = 0$) at which the gradient of the total cost with respect to the control vanishes (see Remark \ref{rem:hu_implicit} in \Cref{sec:the_pontryagin_s_maximal_principle_and_algorithm}).
To avoid this and facilitate the convergence to the global minimum, we choose the initial controls so that all the desired collision will be encountered.
Details are provided in each subsection below.

To describe and quantify the performance of the algorithm,
we define the $L^2$ semi-norm of a square integrable function $l(\cdot):[0,T]\rightarrow \bR$ as,
\begin{equation}
    \vert\vert l \vert\vert_2 = [\frac{1}{T-2\omega}
    (\int_{0}^{t_c - \omega} l^2(t)~dt+\int_{t_c+\omega}^{T} l^2(t)~dt)]^{1/2},
    \label{eq:norm_of_scalar}
\end{equation}
when there is only one collision in the optimal trajectory and $t_c$ denotes the corresponding collision time.
In the integral, we exclude a small interval with length $2\omega$ around the collision time.
When there are multiple collisions, the semi-norm can be defined similarly.
For a vector-valued function $\bh=[h_x,h_y]^T$,
we define its $L^2$ semi-norm as
\begin{equation}
    \vert\vert\bh\vert\vert_2 = (\vert\vert h_x\vert\vert_2^2 + \vert\vert h_y\vert\vert_2^2)^{1/2}.
    \label{eq:norm_of_vector}
\end{equation}
Note that we always use $x, y$ subscripts to denote the $x, y$ coordinates since our experiments are in 2D.
In \eqref{eq:norm_of_scalar} and \eqref{eq:norm_of_vector}, we abuse the notation of $L^2$ norm $\norm[2]{\cdot}$.
We will use the definitions here throughout \Cref{sec:experiments}.

This semi-norm can better measure the distance between the simulated numerical control and the analytical one than the $L^2$ norm. The reason is that a tiny shift of the collision time may introduce a relatively large numerical difference in the controls (and the $L^2$ norm), which nevertheless has negligible effects on the dynamics.
One should take $\omega$ to be larger than the smallest discretization size $\Delta t$. In the experiments, we take $\omega=0.05T$.

In the following experiments, we will always denote the analytical solution for the optimal control by $\bu^{opt}$, the total objective by $J$, and the optimal values of the objective by $J^{opt}$.
We will denote the simulated optimal first collision time by $s$ and its analytical solution by $s^{opt}$. 

In this section, we systematically explore the effectiveness and robustness of our algorithm under various conditions. First, we discretize the problem using $N=480$ and observe that our algorithm reliably converges across a range of scenarios, including those with one or more collisions. Linear convergence is observed in all experiments. To further assess the convergence to analytical optimal solutions, we repeat each experiment with varying discretization sizes for $N=60, 120, \cdots, 1200$. Convergence to analytical solutions is observed in all cases.  Our investigation includes a diverse set of examples: a concentric collision scenario described in \Cref{sub:centric_collision}, a more complex case of a general collision explained in \Cref{sub:general_collision_for_two_balls}, and an instance favoring multiple collisions presented in \Cref{sub:multiple_collisions}. \rev{We also report the performance for a non-convex running cost in \Cref{sub:L1-cost} and the convergence rate with the relaxation parameter $\alpha$ in \Cref{sub:convergence_rate}}. Finally, a comparative analysis with the HMP-MAS method, a direct method, and deep reinforcement learning further underscores the superior performance of our approach, elaborated in \Cref{sub:comparison_to_hmp_mas,sec:comparison_to_direct,sub:comparison_rl}.

\subsection{Concentric collision}%
\label{sub:centric_collision}
We first consider the case when the optimal control is to drive disc 1 to hit disc 2 concentrically. We take the initial state as $\bq_1 = [-2, -2]^T, \bq_2=[-1, -1]^T, \bv_1=\bv_2=[0, 0]^T $ and $\epsilon=0.1, T=1$.
The objective is to push disc 1 to strike disc 2 so that disc 2 will be close to the origin at the terminal time $T=1$.
See Figure \ref{fig:case3_illustration} for an illustration.
\begin{figure}[!htb]
    \centering
    \subfloat[Initial configurations]{
        \label{fig:case3_illustration}
        \includegraphics[width=0.65\linewidth]{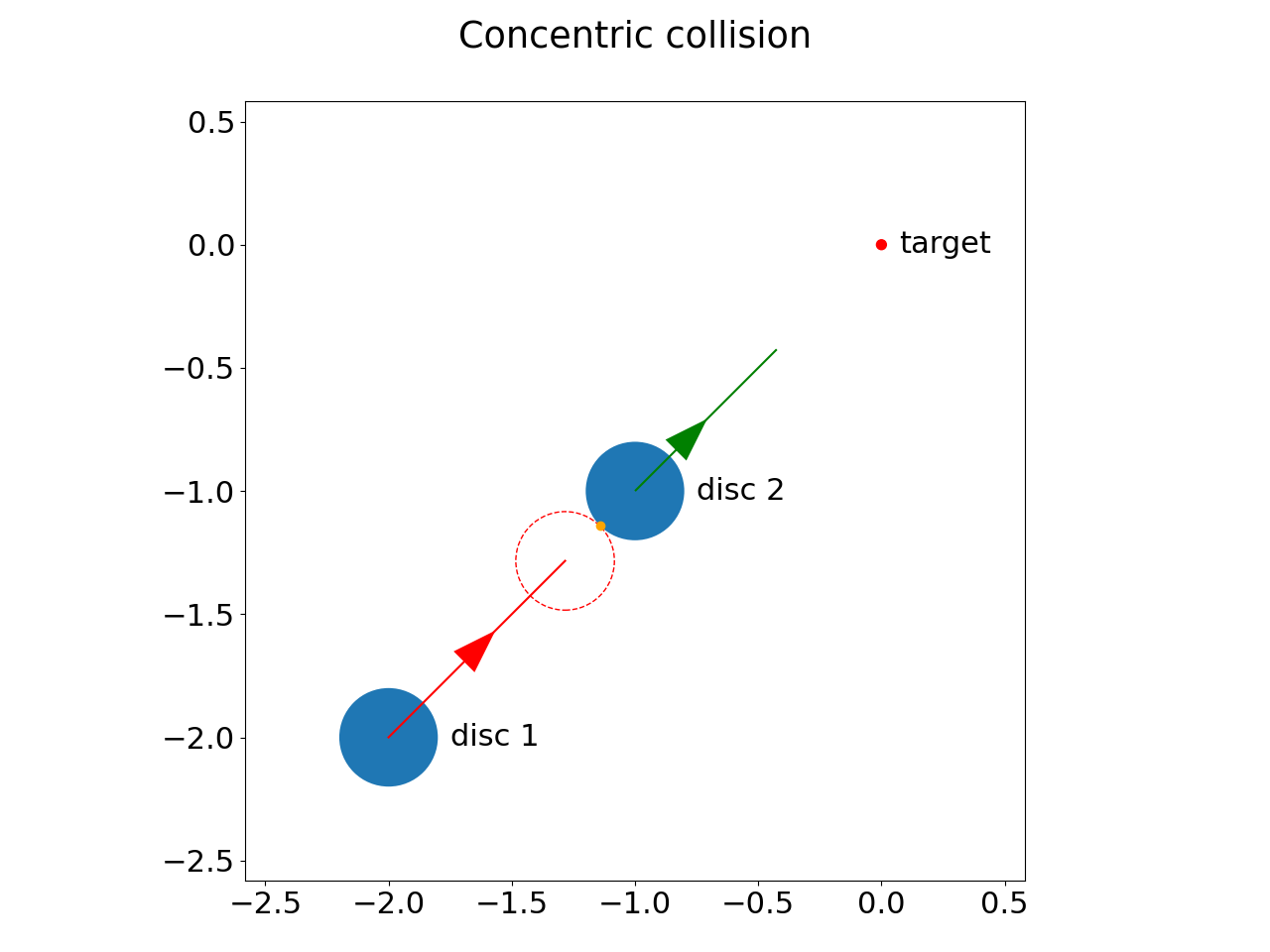}
    }\qquad
    \subfloat[Shape of the optimal control]{
        \label{fig:case3_shape}
        \includegraphics[width=0.95\linewidth]{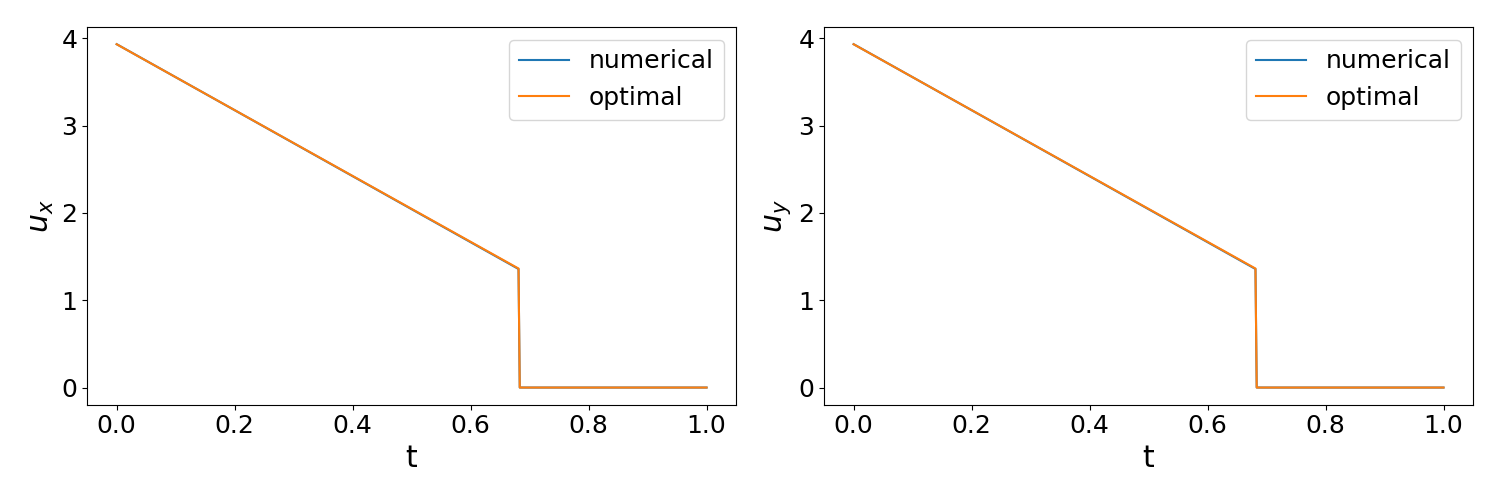}
    }
    \caption{
        Illustration of the concentric collision in \Cref{sub:centric_collision}.
        \protect\subref{fig:case3_illustration}
        Initial state:
        The initial locations of discs are the blue discs.
        The red line and green line are optimal trajectories of disc 1 and disc 2, respectively. 
        The dashed red circle is the before-collision position of disc 1.
        The orange point is the collision point.
        \protect\subref{fig:case3_shape}
        The numerical optimal control and the analytical optimal control, $\bu=[u_x, u_y]$.
    }%
    \label{fig:case3_illustration_shape}
\end{figure}

We discretize the dynamics into 480 steps, \ie $N=480$ and run Algorithm \ref{alg:PMP} with the relaxation parameter $\alpha=0.01$. 
The initial control can be any force fields that make collisions between the two discs happen.
Here, we take a constant force $[3, 3]^T$. 
The optimal analytical control for this example can be solved to machine accuracy; see \Cref{sub:analytical_optimal_controls}.
We demonstrate and compare the simulated numerical control and the analytical optimal control in \Cref{fig:case3_shape}.

In Figure \ref{fig:case3_convergence_versus_iters}, we visualize the solution process by showing how the cost, $\vert\vert H_{u}\vert\vert_2, $ and $\vert\vert\bu-\bu^{opt}\vert\vert_{2}$ converge during the iteration.
We can see that the algorithm converges well in a few hundred iterations. 
Besides, the convergence indicator $\vert\vert H_{u}\vert\vert_2$ implies linear convergence for this example.

\begin{figure}[!htpb]
    \centering
    \includegraphics[width=0.8\linewidth]{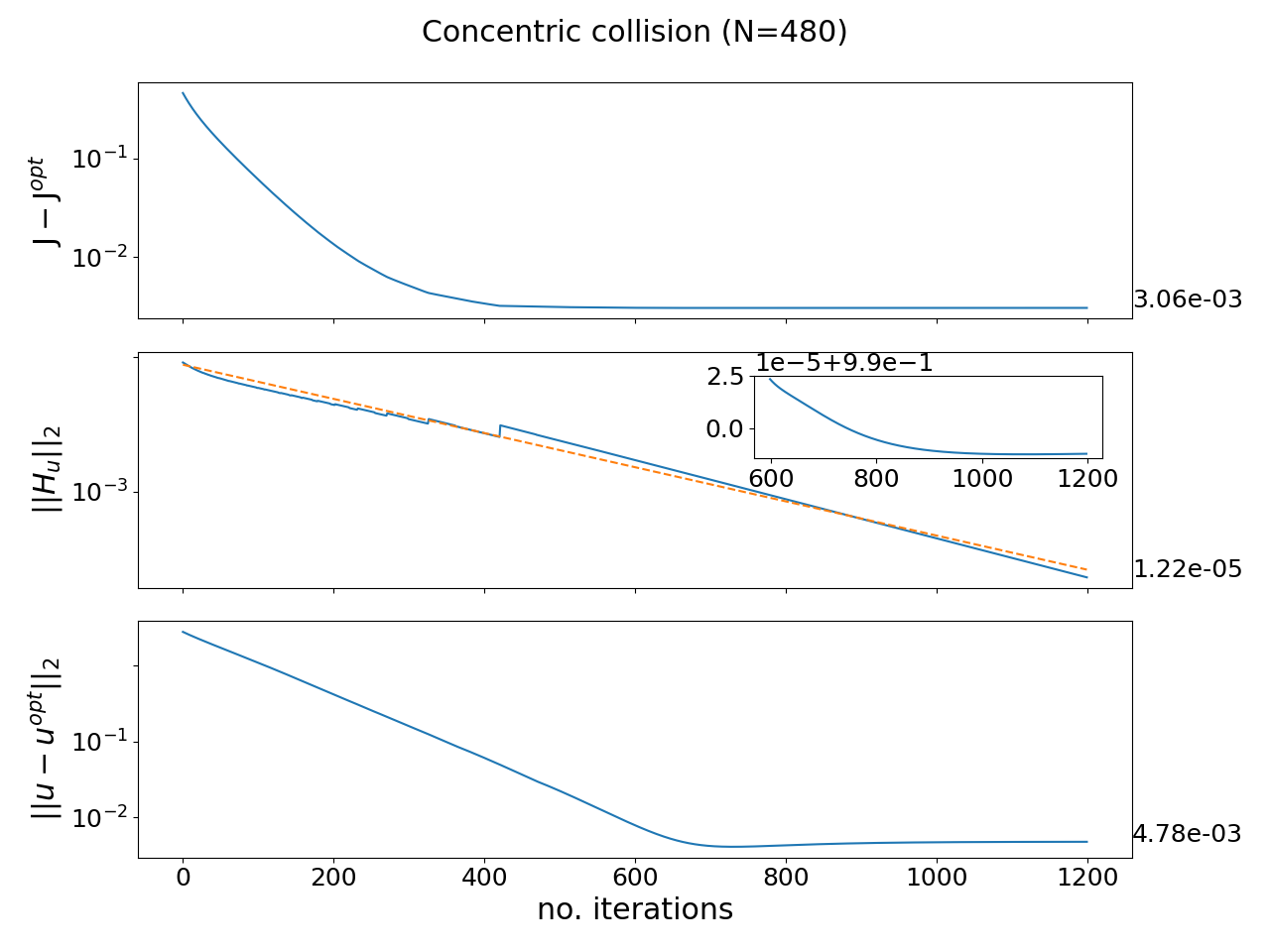}
    \caption{
    Convergence of the iterative algorithm on the concentric collision example in \Cref{sub:centric_collision}.
    The annotations on the right are the final values of the corresponding curves.
    The common $x$ axis denotes the number of iterations spent in the \Cref{alg:PMP}.
    Top: the difference between the analytical optimal objective and the numerical solutions.
    Middle: the $L^2$ semi-norm of the derivative of the Hamiltonian with respect to control $\bu$.
    The dashed line is the result of the regression in the form $\mu\sigma^k$ (this is the linear regression after taking the logarithm.
    $k$ is the number of iterations, $\mu,\sigma$ are positive regression parameters).
    The inset plot is the ratio $\norm[2]{ H_{u}^{k+1}}/\norm[2]{H_{u}^{k}}$ versus the number of iterations. 
    It suggests that the rate of linear convergence is approximately $0.99$.
    Bottom: the difference between the analytical optimal control and the numerical solutions in terms of the $L^2$ semi-norm.}%
    \label{fig:case3_convergence_versus_iters}
\end{figure}

We repeated the experiments for $N=60, 120, \cdots, 1200$, and find that the numerical solution converges to the analytical solution and that the simulated collision times also converge to the analytical solution.
Perfect first-order accuracy is observed; see Figure \ref{fig:case3_convergence_versus_steps}.

\begin{figure}[!htpb]
    \centering
    \includegraphics[width=1.0\linewidth]{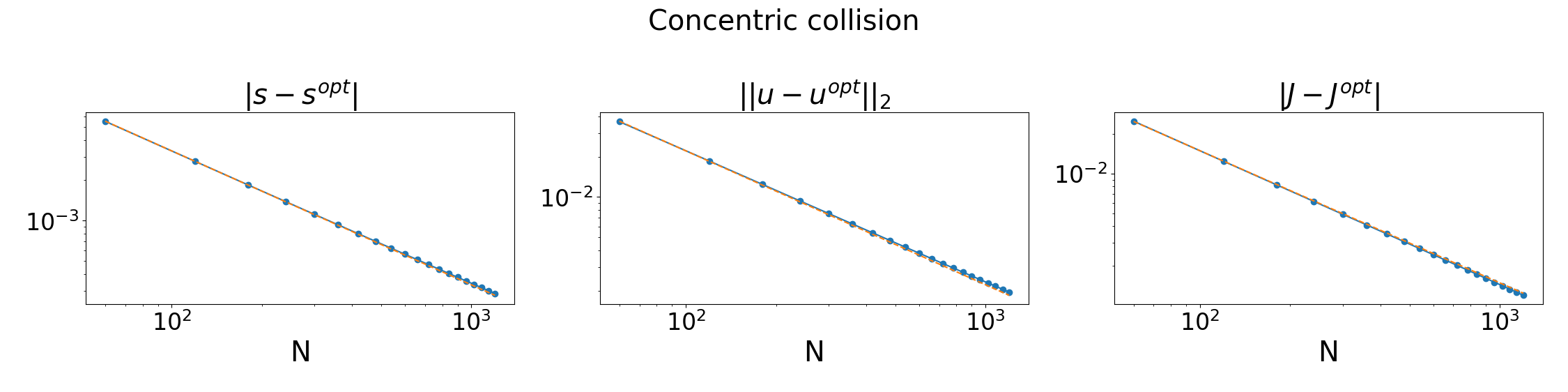}
    \caption{
    Convergence to the analytical solution when the number of time intervals $N$ tends to infinity on the concentric collision example in \Cref{sub:centric_collision}.
    The dashed orange lines are straight lines obtained from the least square regression on the data with model $\sigma/N$ and parameter $\sigma$.
    Left: Difference between the simulated collision times and the analytical optimal collision time.
    Middle: Difference between the numerical and analytical solutions for the optimal controls in terms of the $L^2$ semi-norm.
    Right: Difference between the simulated total costs and the analytical solution.
    }%
    \label{fig:case3_convergence_versus_steps}
\end{figure}

\subsection{General situation for two discs}%
\label{sub:general_collision_for_two_balls}
This experiment shows that our algorithm performs well in more general situations when relative velocity forms a non-zero angle with the axis of collision.
We take the initial state to be $\bq_1=[-1, -2]^T, \bq_2=[-1,-1]^T, \bv_1=\bv_2=[0, 0]^T $ and $\epsilon=0.01, T=1$. See Figure \ref{fig:case3_oblique_illustration} for an illustration of the setup. 

\begin{figure}[!htpb]
    \centering
    \subfloat[Initial configurations]{
        \label{fig:case3_oblique_illustration}
        \includegraphics[width=0.65\linewidth]{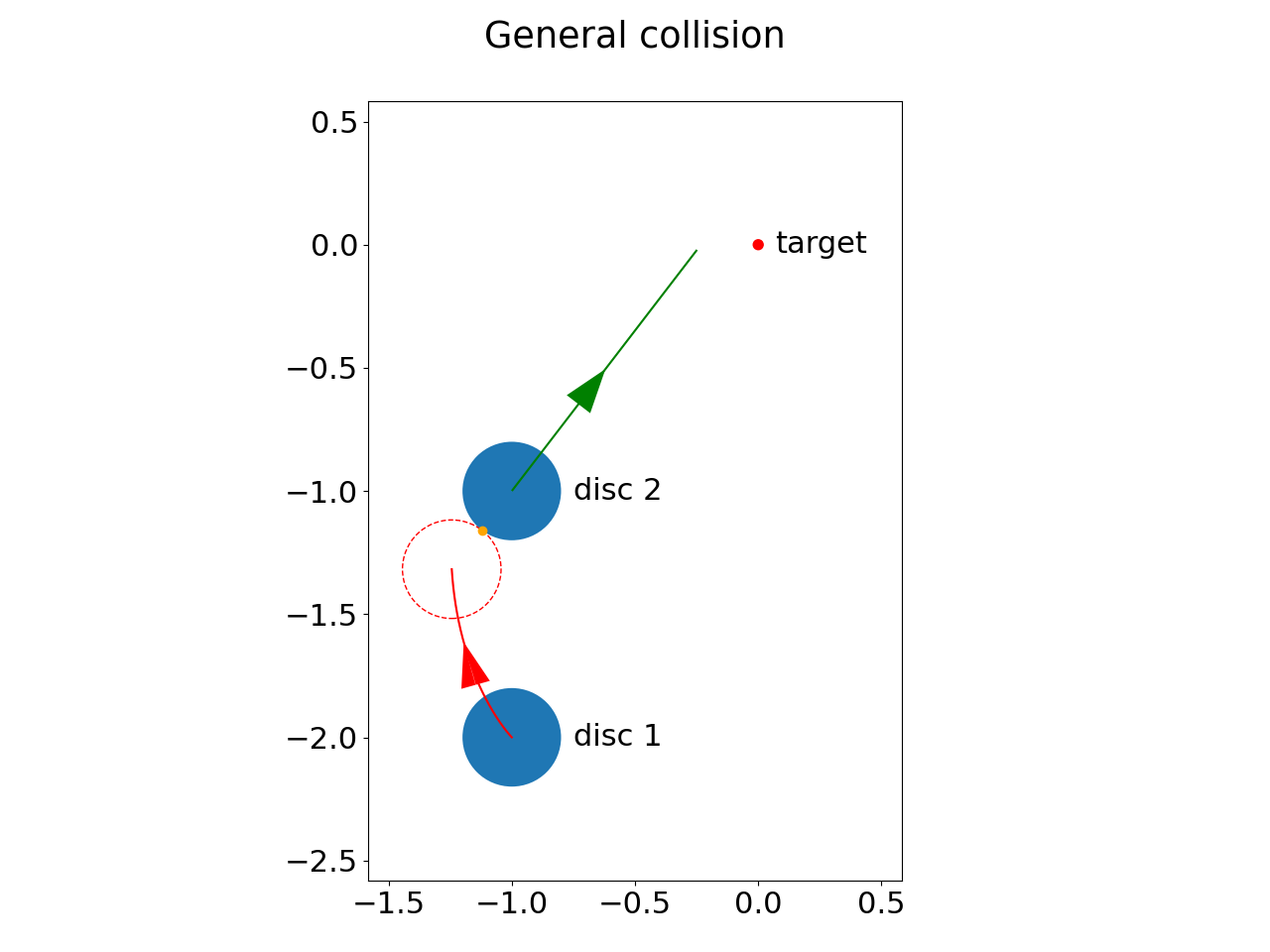}
    }\qquad
    \subfloat[Shape of the optimal control]{
        \label{fig:case3_oblique_shape}
        \includegraphics[width=0.95\linewidth]{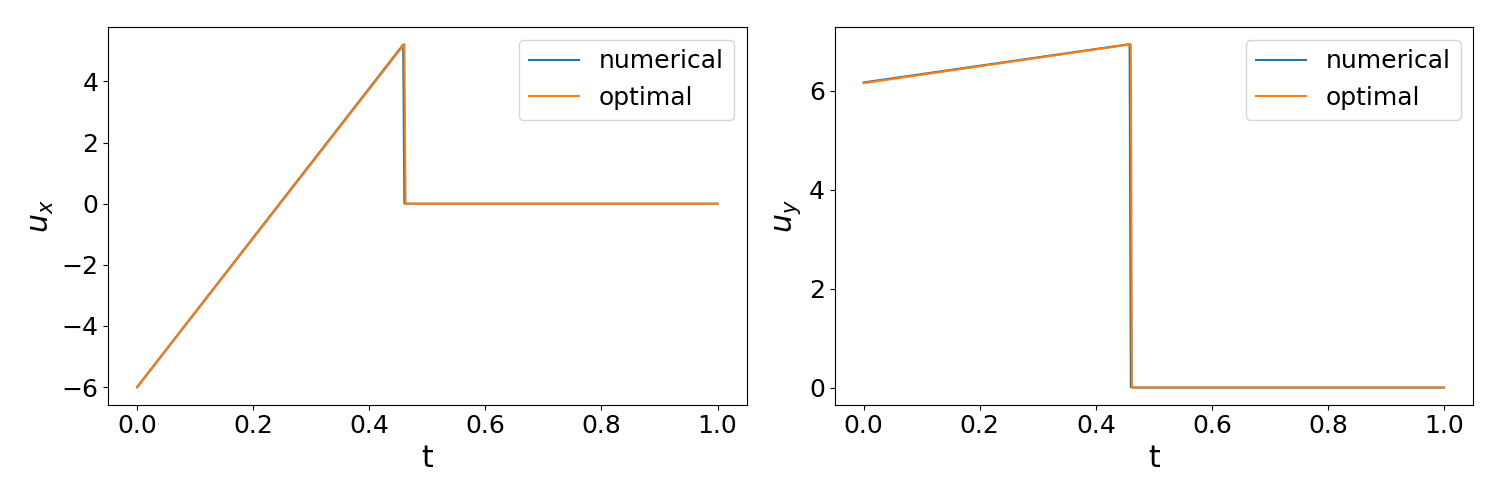}
    }
    \caption{
        Illustration of the general collision in \Cref{sub:general_collision_for_two_balls}.
        The meaning of each plot is similar to that of \Cref{fig:case3_illustration_shape}.
    }%
\end{figure}

We discretize the dynamics into 480 steps, \ie $N=480$ and run Algorithm \ref{alg:PMP} with the relaxation parameter $\alpha=0.01$.
The initial control is a constant field $[0, 3]^T$ under which disc 1 collides with disc 2.
The analytical solution can also be found for this problem,
see Appendix \ref{sub:analytical_optimal_controls} for the derivation and Figure \ref{fig:case3_oblique_illustration} for the optimal trajectories.
The numerical result is shown and compared to the analytical one in Figure \ref{fig:case3_oblique_shape}.

In Figure \ref{fig:case3_oblique_convergence_versus_iters}, we show the change of cost, $\vert\vert H_{u}\vert\vert_2$ and $\vert\vert\bu-\bu^{opt}\vert\vert_{2}$.
We see that a few hundred iterations are sufficient for convergence.
We notice that linear convergence is also observed in terms of the convergence indicator $\vert\vert H_{u}\vert\vert_2$.
\begin{figure}[!htpb]
    \centering
    \includegraphics[width=0.8\linewidth]{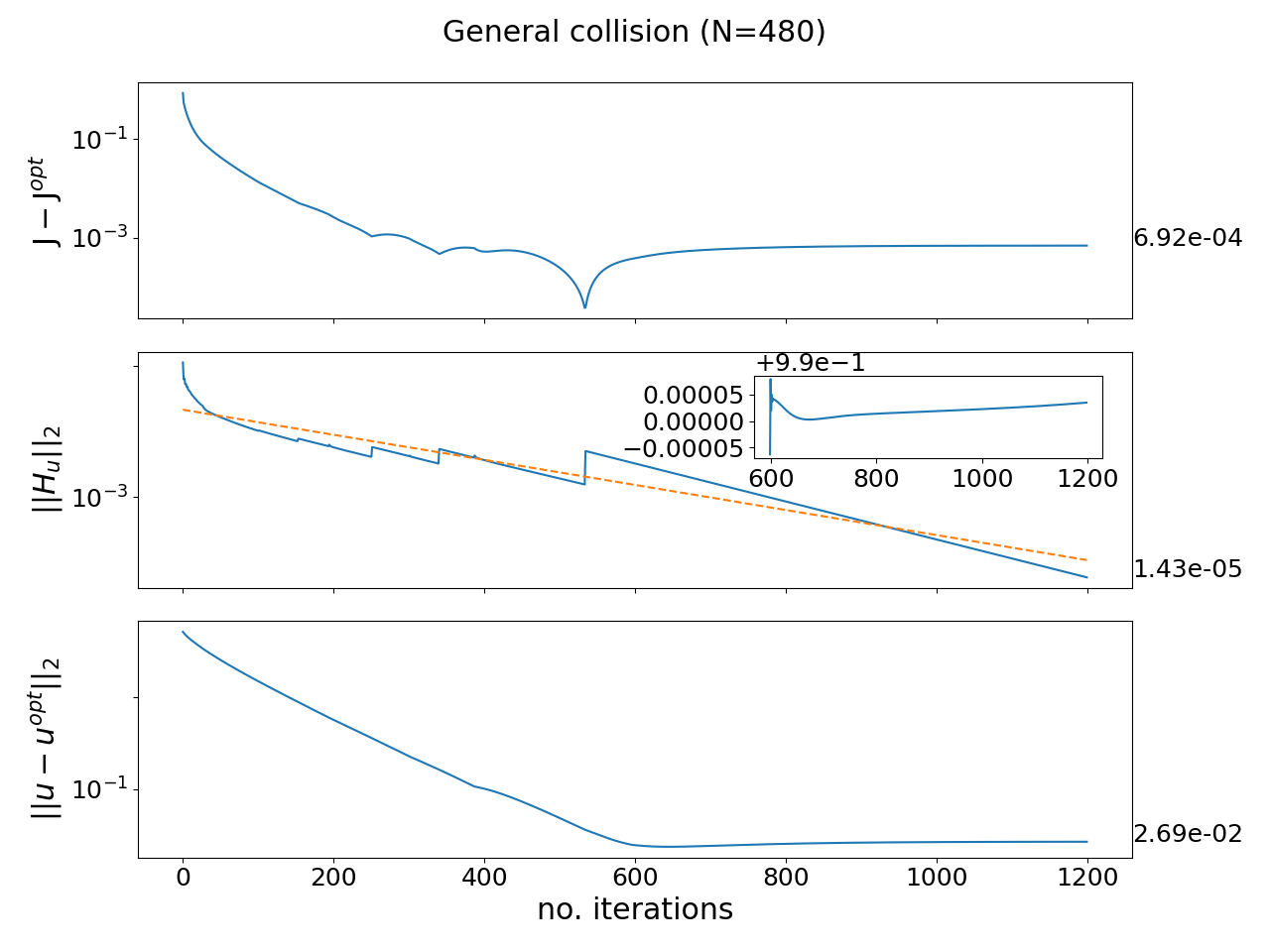}
    \caption{
    Convergence of the iterative algorithm on the general collision example in \Cref{sub:general_collision_for_two_balls}.
    The meaning of each plot is similar to that of \Cref{fig:case3_convergence_versus_iters}.
    }%
\label{fig:case3_oblique_convergence_versus_iters}
\end{figure}

We also demonstrate the convergence to the analytical solution as $N$ increases.
Asymptotic first-order accuracy is observed for this case;
see Figure \ref{fig:case3_oblique_convergence_versus_steps}.

\begin{figure}[!htpb]
    \centering
    \includegraphics[width=1.0\linewidth]{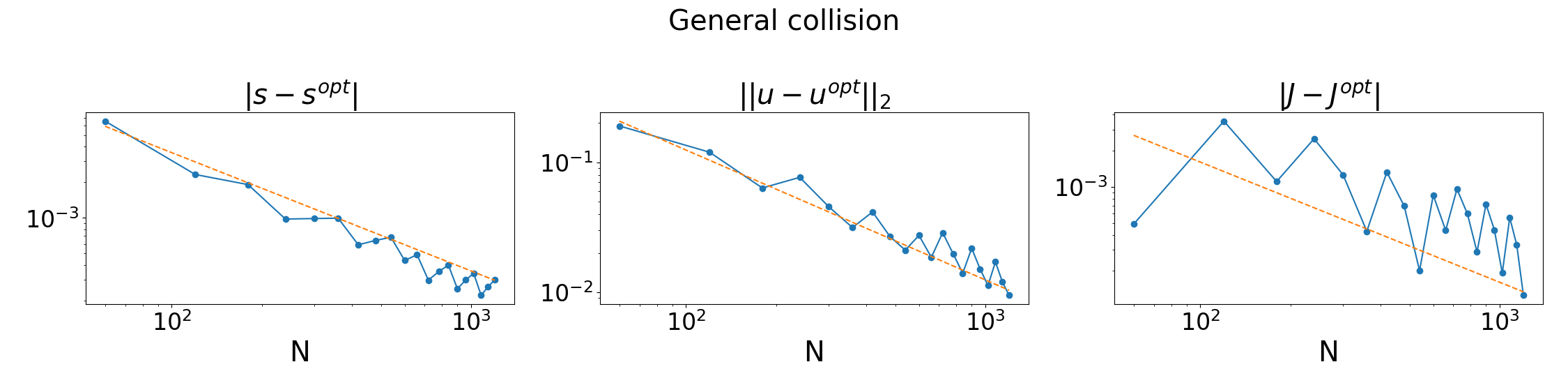}
    \caption{
    Convergence to the analytical solution when the number of time intervals $N$ tends to infinity on the general collision example in \Cref{sub:general_collision_for_two_balls}.
    The meaning of each plot is similar to that of \Cref{fig:case3_convergence_versus_steps}.
    }%
\label{fig:case3_oblique_convergence_versus_steps}
\end{figure}

\subsection{Multiple collisions}%
\label{sub:multiple_collisions}
In this example, we test our algorithm in the case where multiple collisions are needed to achieve optimality.
We add a wall, treated as a rigid body with an infinitely large mass.
There are three possible collisions in this system; inter-disc, the wall with disc 1 and the wall with disc 2.
A natural way to model this system is to independently model the collision, \ie the detection and jump functions, of each possible colliding pair of rigid bodies.
The functions of the wall-disc collision are detailed in \Cref{app:wall_disc_colllision}. Subsequently, we employ the method outlined in \Cref{app:merge_multiple_conllision} to formulate the problem as that in \Cref{sec:formulation_and_assumptions}.

We take the initial state as $\bq_1=[-\frac{3}{2}-\frac{\sqrt{2}}{5}, \frac{1}{10}-\frac{\sqrt{2}}{5}]^T, \bq_2=[-1, \frac{3}{5}]^T,\bv_1=\bv_2=[0, 0]^T$ and $\epsilon=0.01, T=1$.
The wall, whose normal vector is $[0, 1]^T$, is placed at a distance of $1$ from the origin.
See Figure \ref{fig:case4_illustration} for an illustration.
In this setup, we control disc 1 to collide with disc 2 in the way that disc 2 will approach the target with the help of wall collision.
The analytical solution for the optimal control can also be obtained, see Section \ref{sub:analytical_optimal_controls}.

\begin{figure}[!htpb]
    \centering
    \subfloat[Initial configurations]{
        \label{fig:case4_illustration}
        \includegraphics[width=0.6\linewidth]{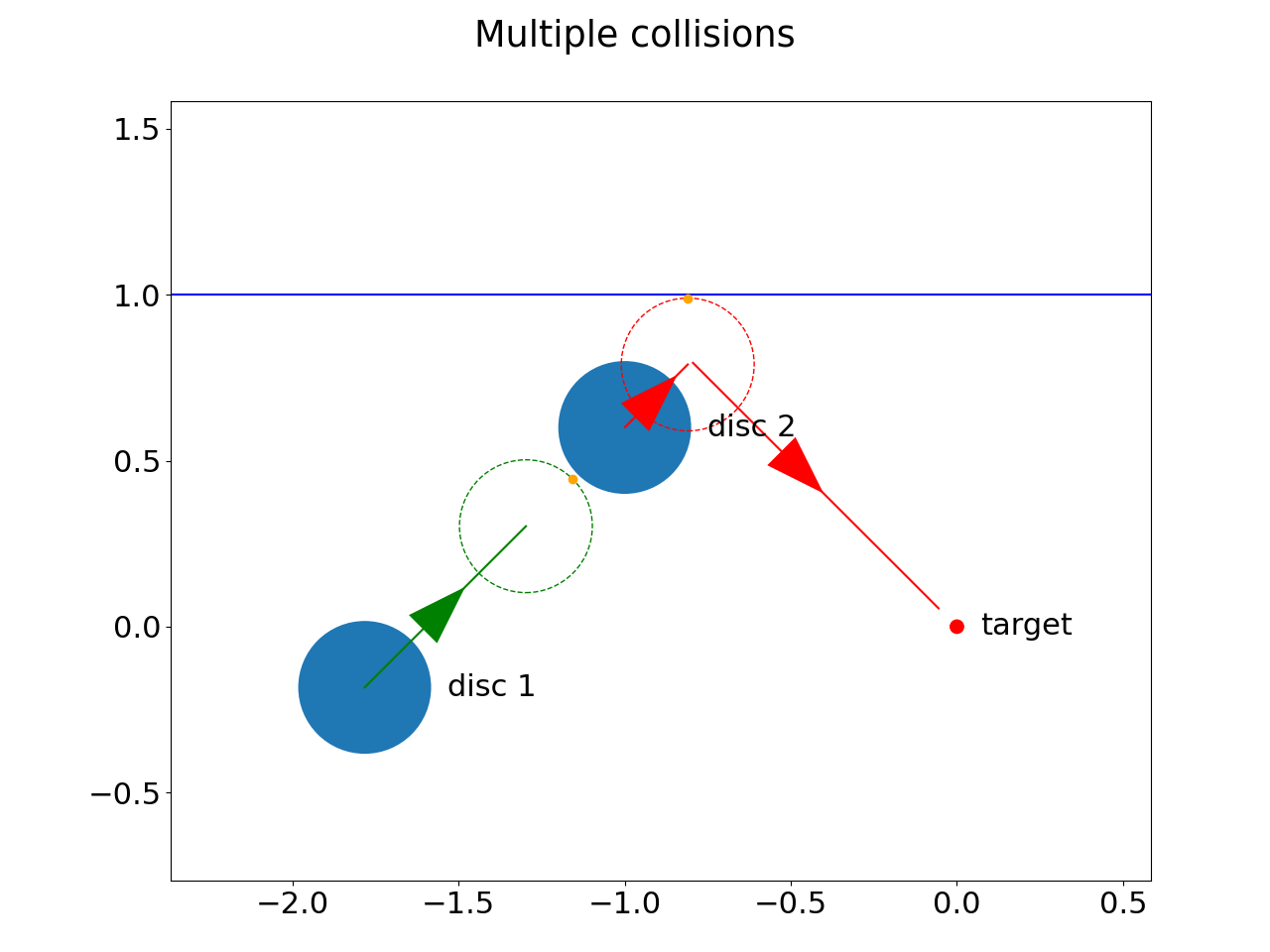}
    }\qquad
    \subfloat[Shape of the optimal control]{
        \label{fig:case4_shape}
        \includegraphics[width=0.95\linewidth]{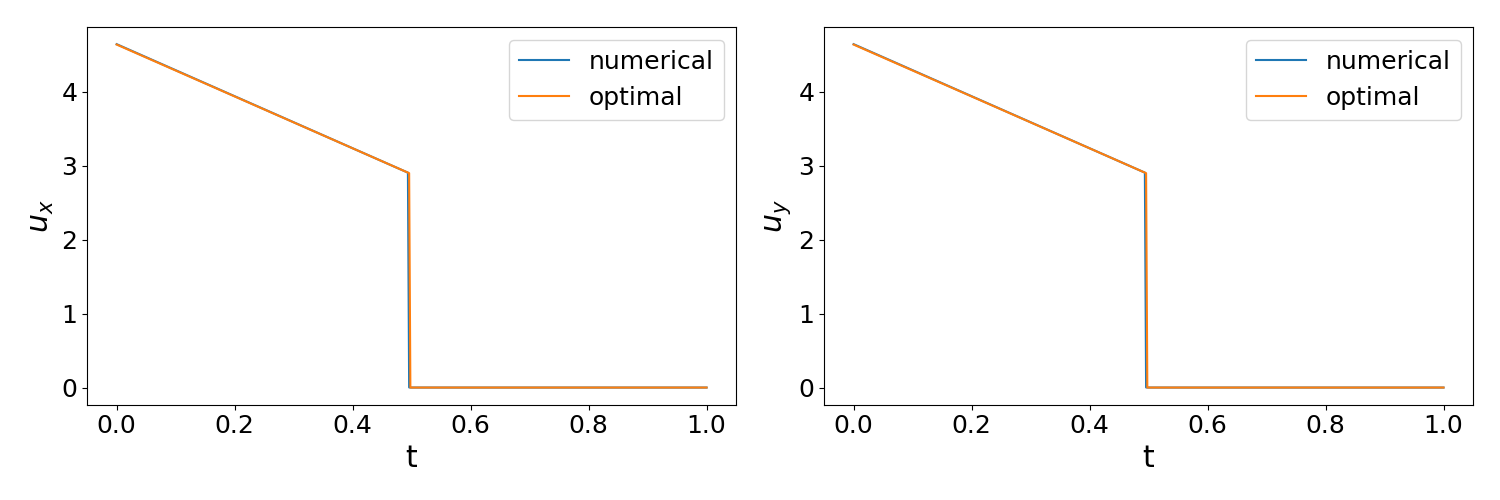}
    }
    \caption{
        Illustration of the multiple collisions in \Cref{sub:multiple_collisions}.
        The meaning of each plot is similar to that of \Cref{fig:case3_illustration_shape}.
        The blue line represents the wall.
        The dashed red circle is the position of disc 2 when disc 2 collides with the wall.
        The orange points are the collision points.
    }%
\end{figure}

We discretize the dynamics into 480 steps, \ie $N=480$ and run the Algorithm \ref{alg:PMP} with the relaxation parameter $\alpha=0.01$.
The initial control is the constant control pointing disc 2 from disc 1 with magnitude $3\sqrt{2}$, \ie $[3, 3]^T$.
The numerical result is shown and compared to the analytical one in Figure \ref{fig:case4_shape}.
Again, linear convergence is observed, and a few hundred iterations are enough for the algorithm to converge, see Figure \ref{fig:case4_convergence_versus_iters}.

\begin{figure}[htpb]
    \centering
    \includegraphics[width=0.8\linewidth]{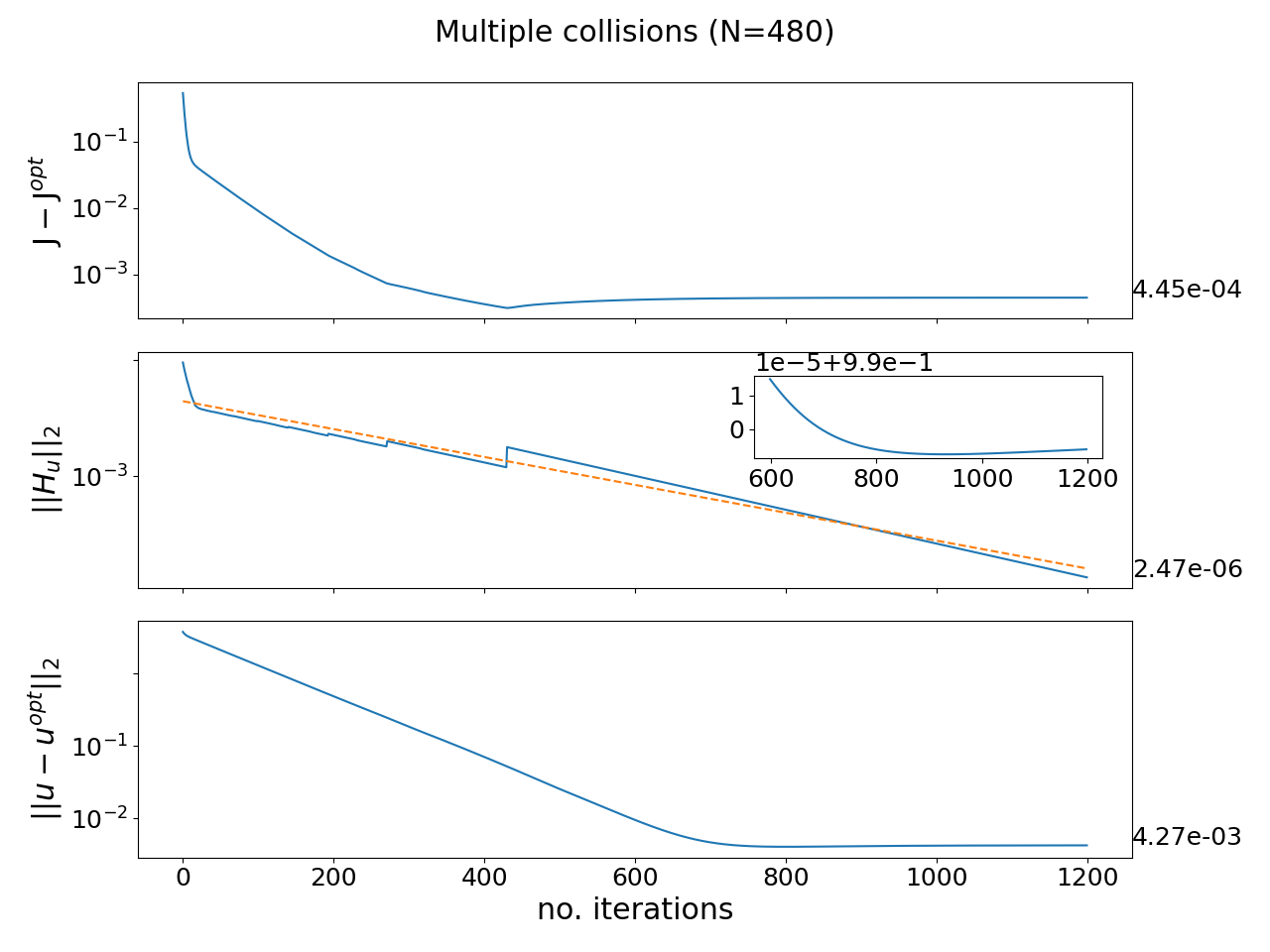}
    \caption{
    Convergence of the iterative algorithm on the multiple collisions example in \Cref{sub:multiple_collisions}.
    The meaning of each plot is similar to that of \Cref{fig:case3_convergence_versus_iters}.
    }%
    \label{fig:case4_convergence_versus_iters}
\end{figure}

As depicted in Figure \ref{fig:case4_convergence_versus_steps}, our optimal numerical control is approaching the analytical solution as $\Delta t$ approaches zero and
the collision times also converge to the analytical ones.
We also note that the order of accuracy is still first order. It does not degrade with more collisions.

\begin{figure}[htpb]
    \centering
    \includegraphics[width=1.0\linewidth]{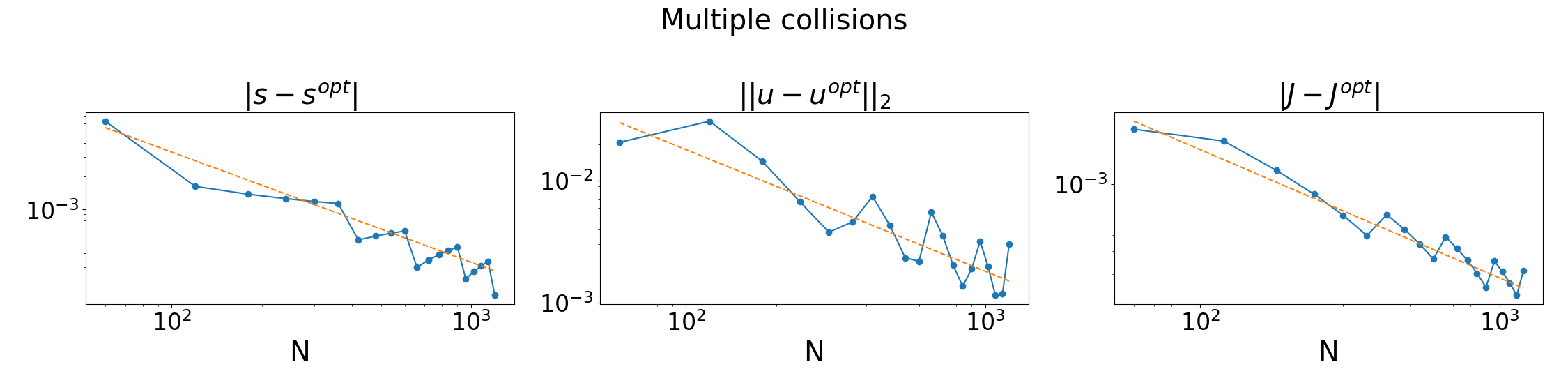}
    \caption{
    Convergence to the analytical solution when the number of time intervals $N$ tends to infinity on the multiple collisions example in \Cref{sub:multiple_collisions}.
    The meaning of each plot is similar to that of \Cref{fig:case3_convergence_versus_steps}.
    }%
    \label{fig:case4_convergence_versus_steps}
\end{figure}

\subsection{\texorpdfstring{$L^1$}{L1} running cost}
\label{sub:L1-cost}
To further demonstrate the performance of the proposed algorithm, we reconsider the example in Section~\ref{sub:general_collision_for_two_balls}, a general collision of two discs, but with a different (non-quadratic) running cost,
\[
L(\bx, \bu) = \epsilon \vert\bu\vert,
\]
under the constraint $\vert \bu \vert \le u_{M}$.
Namely, we have
\[
U = \{(\bu_x, \bu_y)\in\bR^2: \bu_x^2 + \bu_y^2 \le u_M^2 \}.
\]
The associated Hamiltonian is
\begin{equation}
    \label{eq:eg_hamiltonian_l1}
    H(\bx,\blambda, \bu) =  \blambda_{q_1}\cdot \bx_{v_1}+\blambda_{q_2}\cdot \bx_{v_2} + \frac{1}{m_1} \blambda_{v_1}\cdot \bu + \epsilon\vert\bu\vert.
\end{equation}
Its minimizer with respect to $\bu$ is
\begin{equation}
    \label{eq:eg_hamiltonian_l1_argmin}
    \bu_{*} = \arg\min_{\bu \in U} H(\bx,\blambda, \bu) = -u_M \frac{\blambda_{v_1}}{\vert\blambda_{v_1}\vert}
\end{equation}
when $\vert\blambda_{v_1}\vert \ge m_1\epsilon$ and $0$ otherwise. We take $u_M=5\sqrt{2}$ and $\epsilon=0.01$ in the experiment.
The optimal control, numerically solved, is visualized in Figure \ref{fig:gc_l1_norm_shape}.
We also plot the direction of $\blambda_{v_1}$ and $\bu\cdot\blambda_{v_1}/\vert\blambda_{v_1}\vert$.
At the final iteration, $\vert\blambda_{v_1}\vert>m_1\epsilon$ before the collision and $\vert\blambda_{v_1}\vert=0$ after the collision.
The projection $\bu\cdot\blambda_{v_1}/\vert\blambda_{v_1}\vert$ achieves its minimum value $-u_M$ before the collision, which shows that the obtained control minimizes the Hamiltonian \eqref{eq:eg_hamiltonian_l1} according to \cref{eq:eg_hamiltonian_l1_argmin}.
Finally, as depicted in Figure \ref{fig:gc_l1_norm_cost}, the total objective decreases with iterations and converges. We have also tested a different constraint set $U$ that imposes constraints on control componentwisely in this example and the example of concentric collision in Section~\ref{sub:centric_collision}. In all these examples with $L^1$ running cost, the proposed algorithm converges without any further tuning.
\begin{figure}[!htpb]
    \centering
    \subfloat[The optimal control]{
        \label{fig:gc_l1_norm_shape}
        \includegraphics[width=0.95\linewidth]{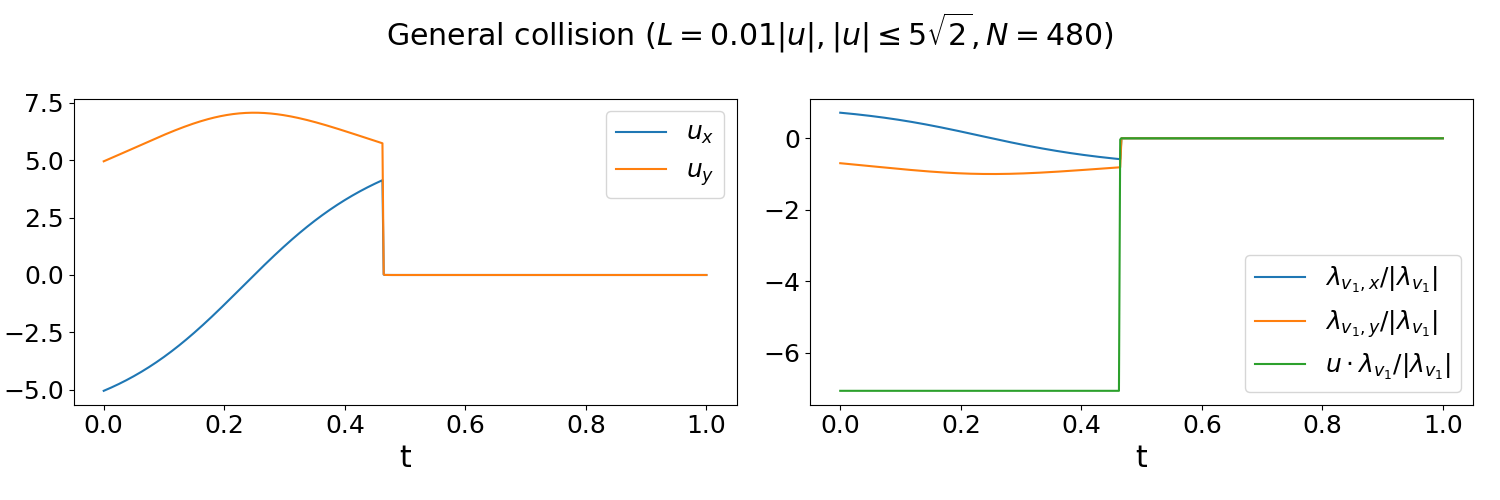}
    }\qquad
    \subfloat[Total costs]{
        \label{fig:gc_l1_norm_cost}
        \includegraphics[width=0.8\linewidth]{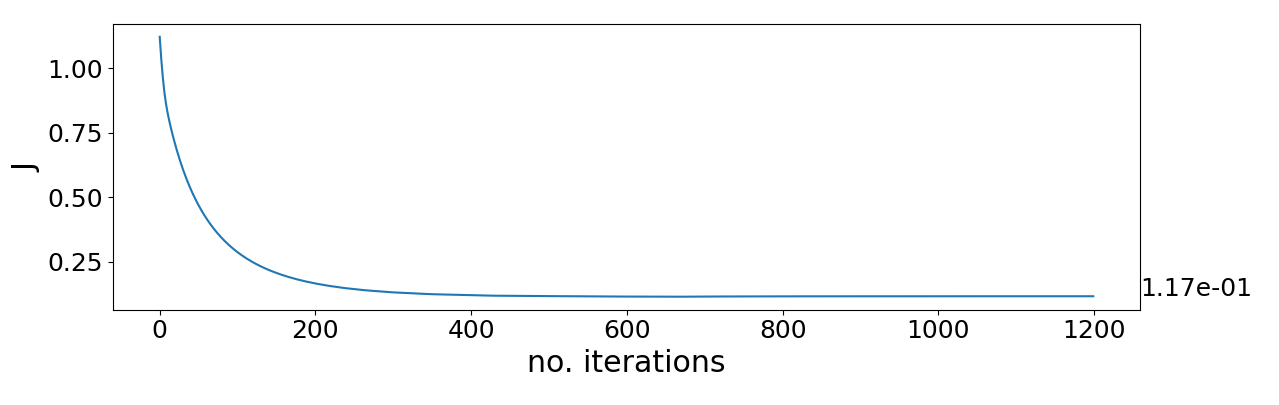}
    }
    \caption{
        Illustration of the general collision with $L^1$ running cost in \Cref{sub:L1-cost}.
        \protect\subref{fig:gc_l1_norm_shape}
        Left: The numerical optimal control $\bu=[\bu_x, \bu_y]$.
        Right: The direction of $\blambda_{v_1}=[\blambda_{v_1, x}, \blambda_{v_1,y}]$ and the projection of $\bu$ to it. 
        \protect\subref{fig:gc_l1_norm_cost}
        The objective versus iterations.
        }
\end{figure}

\subsection{\rev{Convergence rate with $\alpha$}}
\label{sub:convergence_rate}
\rev{
In this subsection, we numerically investigate the convergence rate with respect to the relaxation parameter $\alpha$ in \cref{eq:u_relax_update}, using the general collision example discussed in \Cref{sub:general_collision_for_two_balls}. Specifically, we depicts the convergence of $\bu$ toward the optimal control $\bu^{\text{opt}}$ by tracking the $L^2$ distance $||\bu - \bu^{\text{opt}}||_2$ for three parameter values: $\alpha = 0.01, 0.005$, and $0.001$, clearly demonstrating a linear convergence trend.
To quantify the rate of linear convergence 
    \[
        \lim_{k\rightarrow\infty} \frac{||\bu^{k+1} - \bu^{\text{opt}}||_2}{||\bu^{k} - \bu^{\text{opt}}||_2},
    \]
we fit $||\bu^{k} - \bu^{\text{opt}}||_2$ to the form $\mu \sigma^k$, where $\mu$ and $\sigma$ are fitting parameters. This is equivalent to performing a linear regression on $\log ||\bu^{k} - \bu^{\text{opt}}||_2$. Here, $\sigma$ represents the rate of linear convergence, characterizing how quickly the algorithm converges with respect to the iteration count $k$. For $\alpha = 0.01, 0.005$, and $0.001$, the respective regressed values of $\sigma$ are $0.9911, 0.9956$, and $0.9989$, indicating that $\sigma \approx 1 - \alpha$.
Moreover, we note that for larger values of $\alpha$, such as $0.015$, the algorithm converges to $\bu=0$ without collisions since the large update steps suppress collisions during the iterations. This observation underscores the necessity of incorporating the relaxation parameter $\alpha$ in the standard MSA algorithm to ensure proper convergence.}

\begin{figure}[!htpb]
\centering\includegraphics[width=0.6\linewidth]{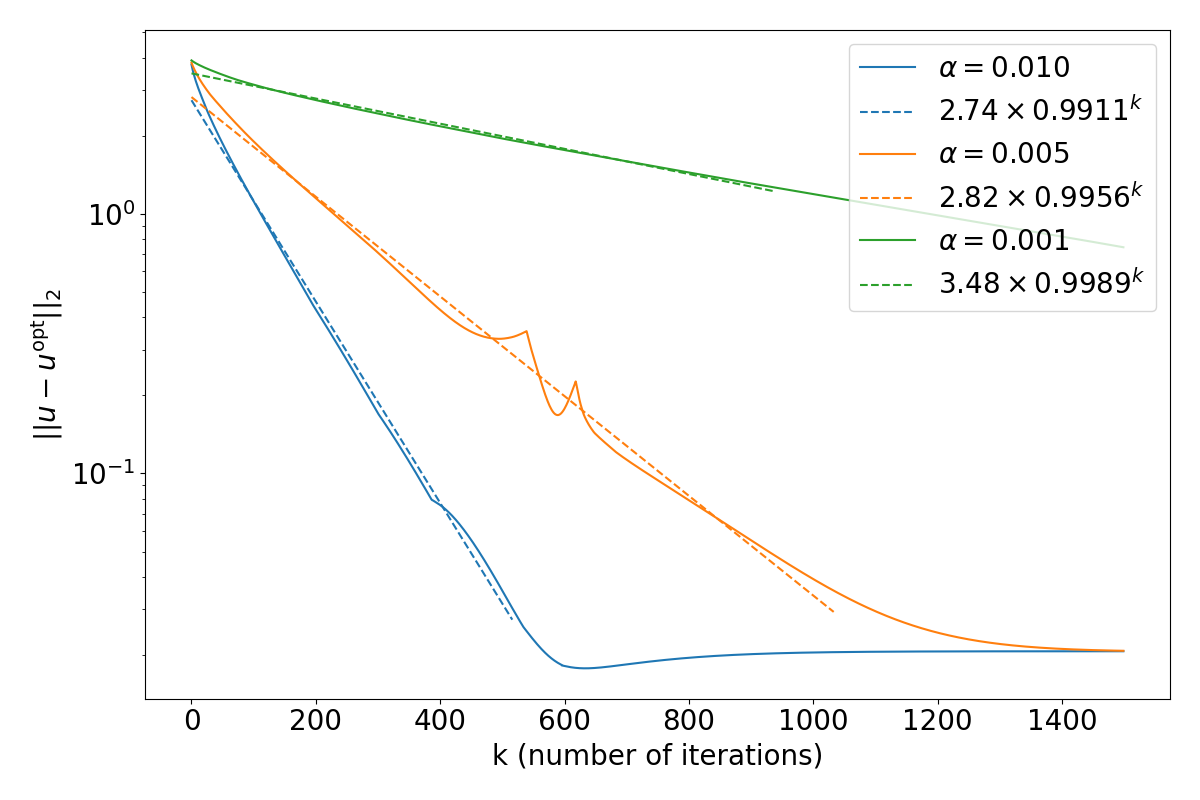}
\rev{
        \caption{Convergence of the iterative algorithm on the general collision example in \Cref{sub:general_collision_for_two_balls} under varying relaxation parameters $\alpha$.
        The dashed straight lines represent the corresponding regression results in the form $\mu\sigma^k$ (linear regression on $\log ||\bu^{k} - \bu^{\text{opt}}||_2$). Here, $k$ denotes the number of iterations, while $\mu$ and $\sigma$ are positive regression parameters.
        The values of $\sigma$ suggest the corresponding rate of linear convergence, approximately satisfying $\sigma \approx 1 - \alpha$.
        }}
        \label{fig:general_collision_different_alpha}
    \end{figure}

%% file: hmp-6-comparisons.tex
\section{Comparisons to other methods}
\label{sec:comparison}

\subsection{Comparison to HMP-MAS}%
\label{sub:comparison_to_hmp_mas}

In \cite{autoSplit}, given the number of switching times, the author proposed to solve the optimal control in multiple autonomous switchings system based on the hybrid maximum principle (HMP-MAS) by alternating between
(1) solving for optimal controls in the sub-domains divided by guessed switching time and switching states;
(2) computing gradients and updating the guesses by gradient descent.
The method has been extended to autonomous hybrid dynamical systems on manifolds \citep{autoSplitManifold, autoSplitManifold2}.
Recently, it has been extended to the autonomous hybrid dynamical systems with both state jumps and switches in dynamics in \cite{hmp_mas_extend2}.
We adapt this method to the dynamical systems discussed in this paper.

To be specific, we assume that there will be one collision for the optimal control.
Starting from a reasonable guess that the collision happens at $t=s$ and at state $\bx_s^-$, we solve two standard optimal control problems:
find the minimal cost control that shoots $(s, \bx_s^-)$ from $(0, \bx_0)$  and 
find the optimal control that minimizing the running cost and terminal cost starting from $\bx_s^+$ at  $t=s$.
Then, one can compute the gradient of the total cost with respect to $s$ and  $\bx_s^-$ and then update them by gradient descent.
The constraint $\phi(\bx_s^-) = 0 $ is enforced as a penalty added to the total costs.

The overall performance of this algorithm is comparable to our algorithm, exhibiting a similar convergence curve and error.
However, one needs to prescribe the correct number of collisions beforehand for this method.
We run the multiple collisions experiment (\cref{sub:multiple_collisions}) starting from an initial trajectory with three collisions (see Figure \ref{fig:case4_multiple_illustration}).
Our algorithm converges to the optimal control with two collisions without any modifications. In contrast, HMP-MAS cannot converge to the optimal solution starting from such an initialization. 
\begin{figure}[htb]
    \centering
    \includegraphics[width=0.65\linewidth]{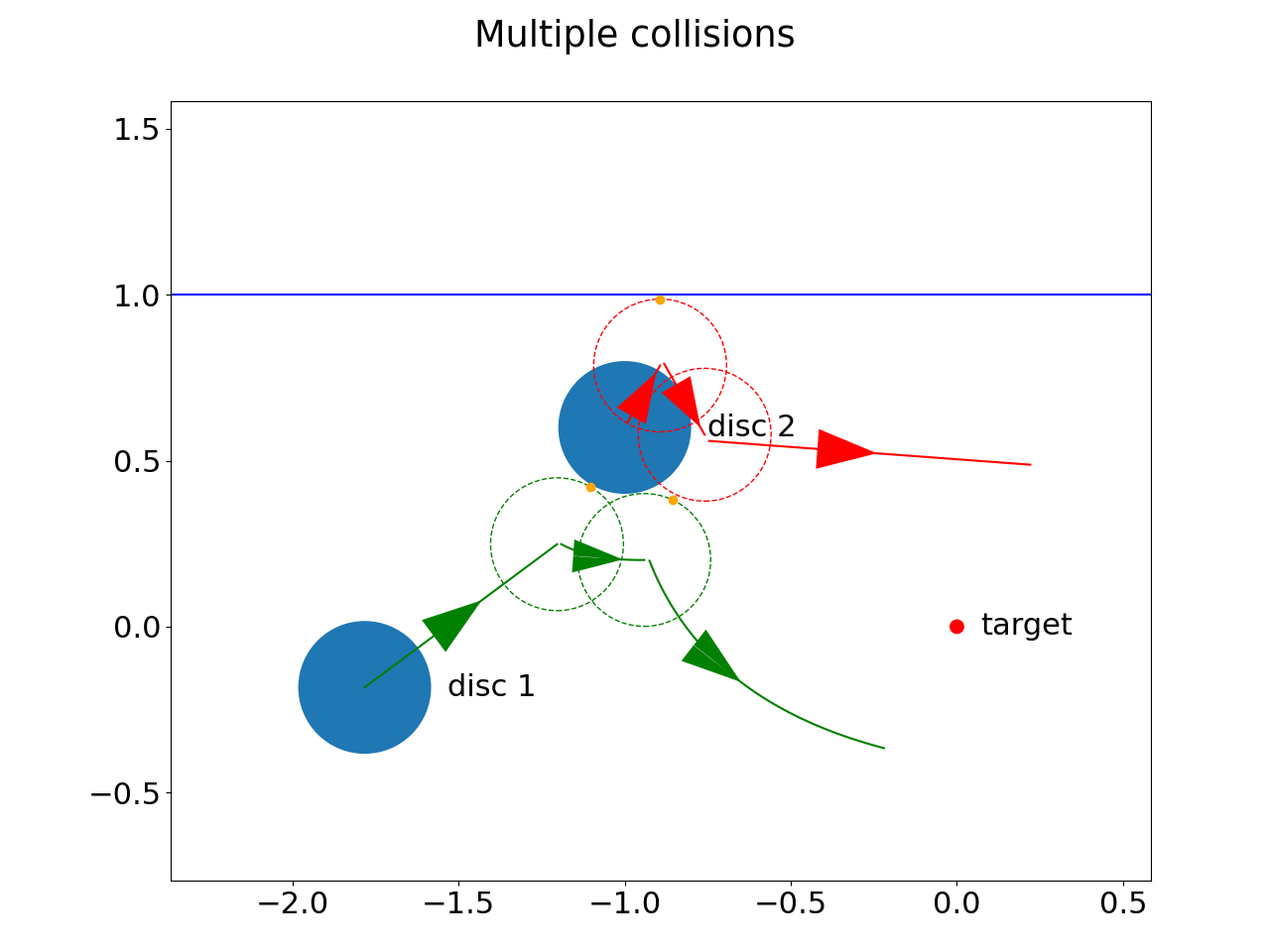}
    \caption{
        Illustration of the trajectories under a control that results in three collisions.
        The green and red lines with arrows are trajectories of disc 1 and disc 2, respectively.
        Dashed green circles and red circles are snapshots of disc 1 and disc 2 when collisions occur, respectively.
        The orange points are collision points.
        See \Cref{sub:comparison_to_hmp_mas}.
    }%
    \label{fig:case4_multiple_illustration}
\end{figure}

\subsection{Comparison to a direct method}
\label{sec:comparison_to_direct}
In this subsection, we compare Algorithm~\ref{alg:PMP} with a gradient descent algorithm, which is applied to the following discretized problem
\begin{mini!}|l| {u_0, \cdots, u_{N-1}}{\phi(\bx_N) + \sum_{i=0}^{N-1}L(\bx_i, \bu_i) \Delta t,} 
    {\label{eq:opt_problem_dis_col}}{}
    \addConstraint{\cref{eq:fd_case1,eq:fd_case2}.}
\end{mini!}
The discrete dynamics is implemented in PyTorch \citep{pytorch} with gradients with respect to the controls $\bu_0, \cdots, \bu_{N-1}$ been calculated by auto-differentiation.

The learning rate is chosen to be compatible with Algorithm \ref{alg:PMP}.
For the examples considered in this paper, the continuous version of the control update step in Algorithm \ref{alg:PMP}, \ie \cref{eq:u_relax_update}, can be rewritten as
\[
\bu^{k+1} = (1-\alpha)\bu^k + \alpha \bu_*  = \bu^k - \frac{\alpha}{2\epsilon}H_{u}(\bx^k, \blambda^k, \bu^k),
\]
recalling the minimizer $\bu_*$ in \cref{eq:eg_u_minimizer} and $H_{u}$ in \cref{eq:eg_h_u}. 
Therefore, we set the learning rate to $\frac{\alpha N}{2\epsilon}$, where the factor $N$ is multiplied to compensate the integration step size $\Delta t$ 
since the discrete gradient has an additional time step factor $\Delta t$ compared to its continuous counterpart. \footnote{
We take the functional $\cS(v) = \int_{0}^{T} v(t)^2\,dt$ as an example to explain why discretization introduces the factor $\Delta t$ in the gradient. Here $v$ is a continuous square-integrable function.
The Fr\'echet functional derivative of $\cS(v)$ in $L^2$ is $S'(v)=2v$.
If we take the discrete approximation of the integral as $\tilde{S}(v_0, \cdots, v_{N-1})=\sum_{i=0}^{N-1}v_i^2\Delta t$, where $v_i$ are values at $t_i$,
the function $\tilde{S}: \bR^{N}\rightarrow\bR$ has gradient $\nabla \tilde{S}(v_0,\cdots,v_{N-1}) = 2\Delta t(v_0,\cdots,v_{N-1})$.}
We note that selecting a consistent learning rate is only for the purpose of comparison.
As can be seen in Figure \ref{fig:direct_performance}, the learning curves of these two methods overlap at the beginning of the iteration.
Further changes to the learning rate and more training iterations will not lead to a reduction in the error of the solution for the direct method.

We have compared the performance of Algorithm \ref{alg:PMP} and the direct method at $N=480$ on all three examples considered previously, and the comparison is shown in Figure~ \ref{fig:direct_performance}.
We find that our HMP-based algorithm always performs much better than the direct method;
the numerical solution of the optimal control and the corresponding total cost of the Algorithm \ref{alg:PMP} are much closer to their analytical ones.

\begin{figure}
    \centering
    \includegraphics[width=1.0\linewidth]{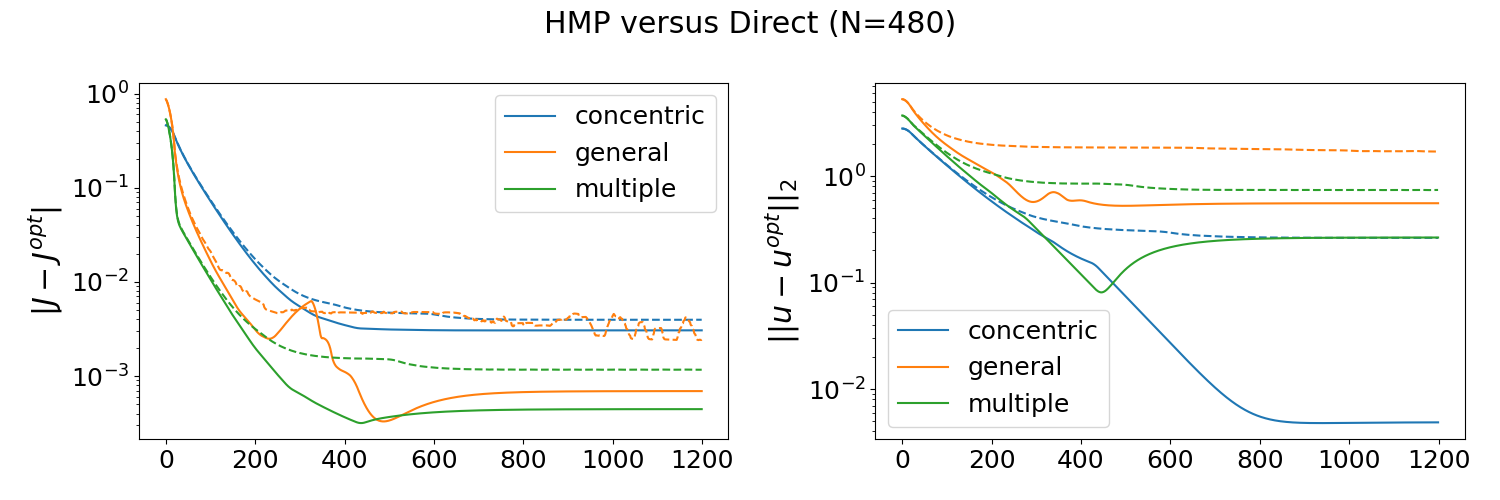}
    \caption{
    The comparison between HMP (Algorithm \ref{alg:PMP}) and the direct method over all three experiments; the concentric collision (\cref{sub:centric_collision}), the general situation (\cref{sub:general_collision_for_two_balls}) and the multiple collisions (\cref{sub:multiple_collisions}).
    Solid line are data from the Algorithm \ref{alg:PMP}.
    Dashed lines are data from the direct method; 
    Different colors correspond to different experiments, see legends in the figure.
    Left: Difference between the simulated total costs and the optimal analytical total costs. Right: Difference between the numerical controls and analytical optimal controls in terms of $L^2$ semi-norm.
    See \Cref{sec:comparison_to_direct}.}
    \label{fig:direct_performance}
\end{figure}

The convergence behavior as $N\rightarrow\infty$ is also compared, see Figures in \ref{fig:direct_compare}.
As seen from the figure, the direct method also produces reasonable solutions in terms of the total costs but does not converge well;
the errors to the optimal analytical controls measured in $L^2$ semi-norm show no signs of convergence as $N$ increases. 
Similar behavior has been observed in the collision times.
The direct method produces comparable results regarding the total costs in the concentric collision example and multiple collisions example but still fails to converge in the general situation example.

\begin{figure}[htpb]
    \centering
    \includegraphics[width=1.0\linewidth]{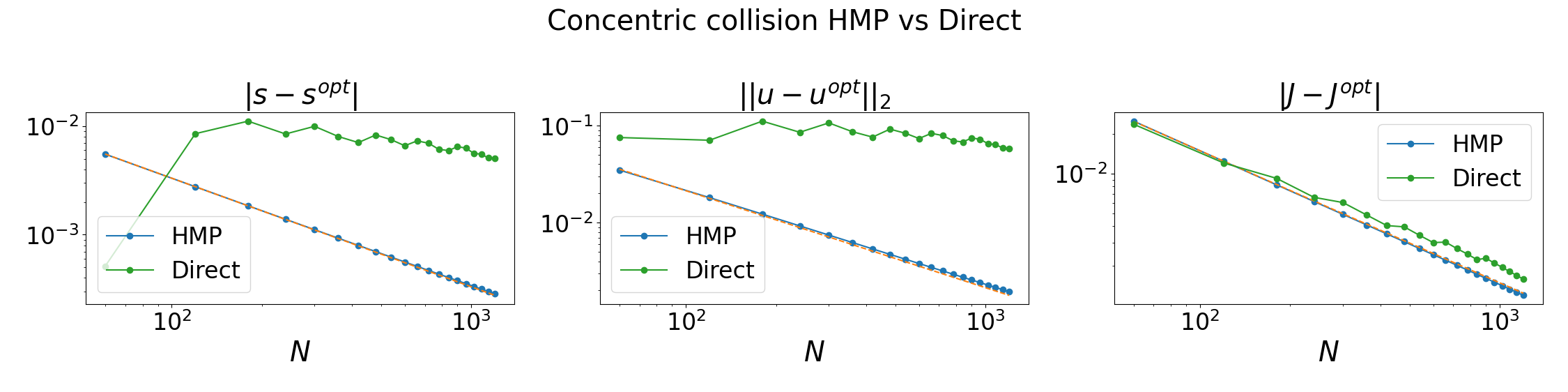}
    \includegraphics[width=1.0\linewidth]{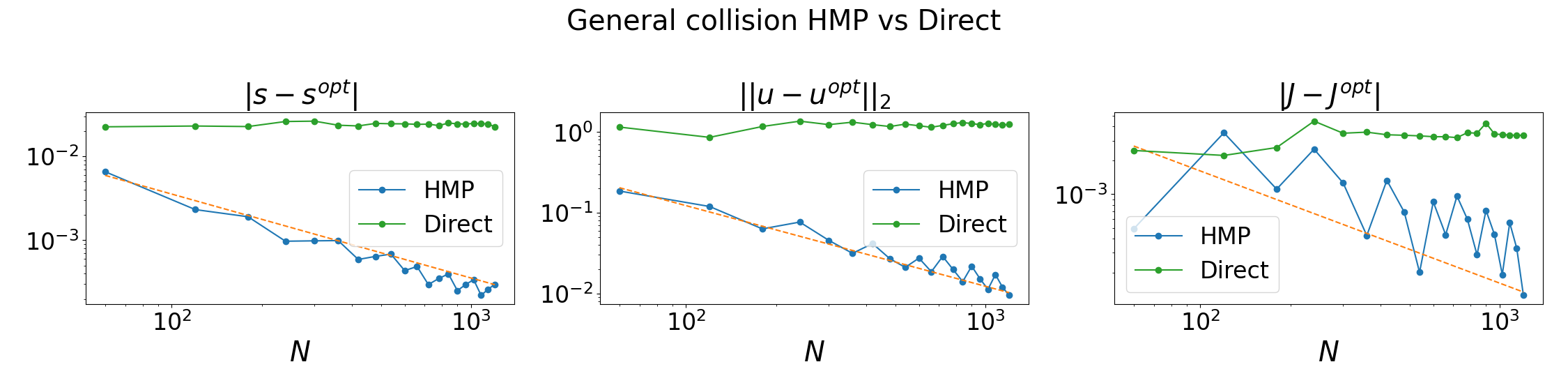}
    \includegraphics[width=1.0\linewidth]{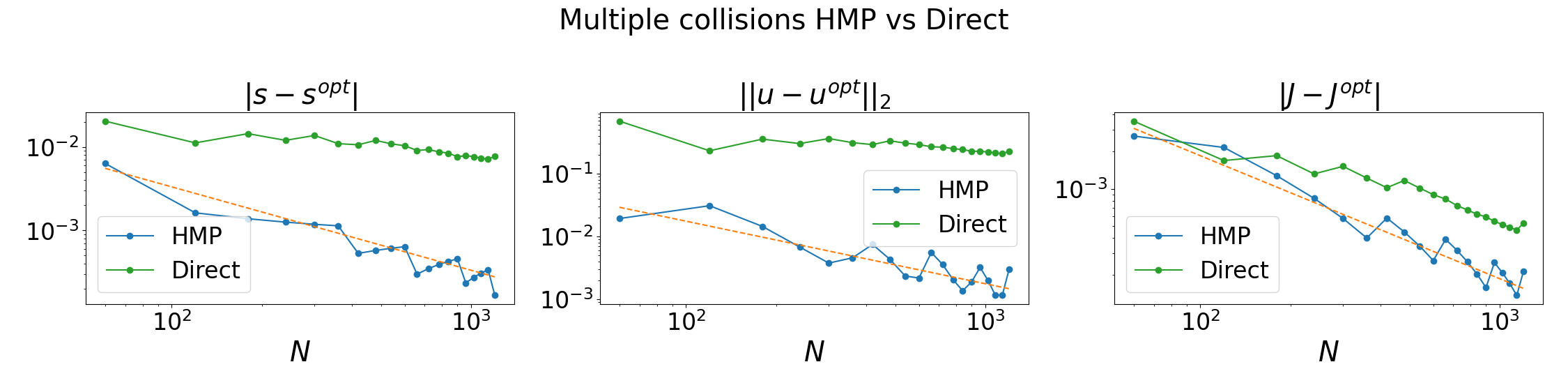}
    \caption{
    Convergence to the analytical solution as $N$ increases.
    The dashed orange lines are straight lines obtained by the least square regression on the results of Algorithm \ref{alg:PMP} with model $\sigma/N$ and parameter $\sigma$.
    Top: The concentric collision example (\cref{sub:centric_collision}).
    Middle: The general situation example (\cref{sub:general_collision_for_two_balls}).
    Bottom: The multiple collisions example (\cref{sub:multiple_collisions}).
    Left: Difference between the simulated collision times and the analytical optimal collision times. 
    Middle: Difference between the numerical controls and the analytical optimal controls in terms of $L^2$ semi-norm.
    Right: Difference between the simulated total costs and the optimal analytical solutions.
    See \Cref{sec:comparison_to_direct}
    }%
    \label{fig:direct_compare}.
\end{figure}

Despite the fact that their gradients are related,  the HMP algorithm and the gradient descent algorithm are still quite different. Numerically, the difference can be seen in 
Figures~\ref{fig:direct_performance} and \ref{fig:direct_compare}. 
Conceptually, for general forms of the dynamics and cost functional, the update direction in the HMP algorithm based on the stronger minimization conditions (\cref{eq:minimal_condition}) is not parallel to the gradient direction. As pointed out in \cite{discontinuous_system,hu2019difftaichi}, the gradient of the discrete problem might be incorrect (in the sense that it does not converge to that of the continuous problem).
In some cases (like the examples considered in this paper) where the minimization direction coincides with the gradient direction, due to the discontinuities, the numerical approximation of the Fr\'echet derivatives $H_{u}$ (of the HMP algorithm, see Remark \ref{rem:hu_implicit} and \cref{eq:hu_dis}) are different from the gradients of the discrete problems (of the direct method, see \cref{eq:opt_problem_dis_col}).
Another point worth noting is that in the HMP algorithm, we can stabilize the numerical computations of $H_{u}$ (\eg \cref{eq:control_resolve_steepest,eq:control_resolve_stable} in the discretization of the backward dynamics) from the continuous point of view. 
As discussed in Section \ref{sec:discretization}, this is of critical importance for the dynamics with discontinuities. In contrast, it is unclear how to adopt similar stabilization techniques for the direct method.

\subsection{Comparison to deep reinforcement learning}
\label{sub:comparison_rl}
Deep reinforcement learning has demonstrated its ability in many complex tasks \citep{RL_dqn,RL_continuous_control,RL_go,RL_go_zero,RL_stack}. In this section, we reformulate the optimal control problem considered in this paper into a reinforcement learning problem and test the performance of a modern advanced deep reinforcement learning algorithm: the Proximal Policy Optimization (PPO) algorithm~\citep{ppo}.

We take the example of the concentric collision in section~\ref{sub:centric_collision}. The observation space is the time-state space $\bR^9 = (t, \bx), t\in[0,T], \bx\in\bR^8$.
The action space is the two-dimensional control whose components are scaled to $[-1, 1]$,
\[ \bar{\bu} = \frac{\bu +5 }{10}.\]
Here, we assume that each component of the actual control takes values in $[-5,5]$.
We use Stable-Baselines 3 \citep{sb3} to implement the PPO algorithm.
We use two-layer neural networks for the actor and critic networks; each layer has 64 neurons.
Each episode in the training consists of $N=60$ time steps.

As discussed previously, due to the control cost, the randomly initialized policy network will quickly converge to zero.
It is necessary to encourage exploration by using a warm-up step.
This is done by setting the running cost $L$ to zero in the early training stage.
We train the warm-up stage for $3000$ episodes.
The policy network starts converging after a series of actions that triggers a collision is exploited.
See the inset of the right subfigure of Figure \ref{fig:rl_results}.
Afterwards, we restore the running cost defined in \cref{eq:eg_running_cost} and train for another $80000$ episodes.

The convergence history of the rewards and the profile of the best control policy are shown in Figure \ref{fig:rl_results}.
We observe that the best policy is obtained at around $3000+60000$ episodes.
After that, the control network starts converging to zero.
The best policy found is still far from optimal.
The optimal total cost found by the PPO algorithm is $1.617140$, which is $15.80\%$ suboptimal to the optimal cost of $1.396542$.
Note that the optimal cost found by Algorithm \ref{alg:PMP} is $1.421552$.

\begin{figure}[htpb]
    \centering
    \includegraphics[width=0.48\linewidth]{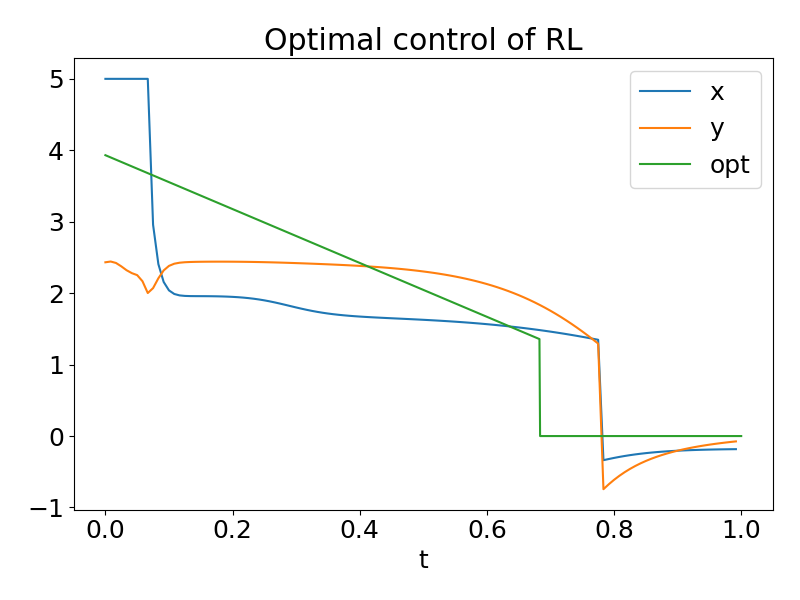}
    \includegraphics[width=0.48\linewidth]{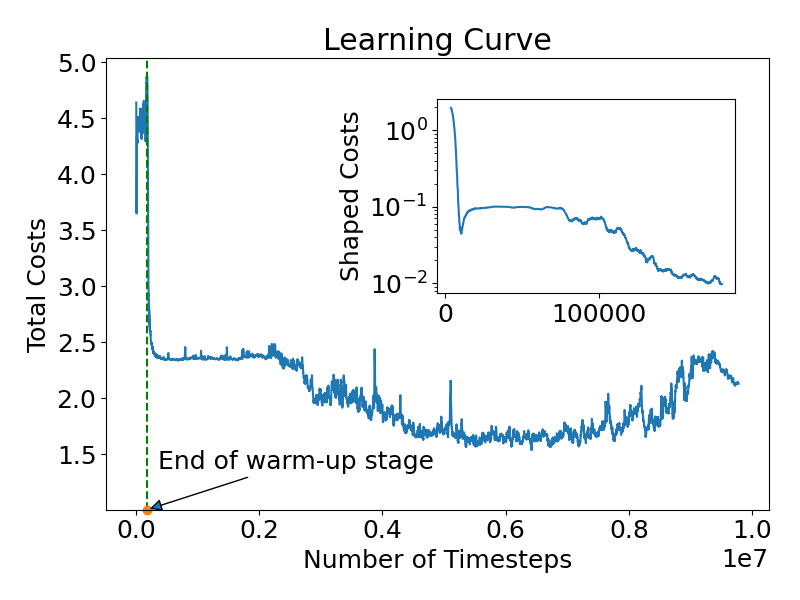}
    \caption{The performance of the deep reinforcement learning algorithm (PPO) described in \Cref{sub:comparison_rl} on the example of concentric collision. Left: the optimal control found by the PPO.
    Curves labelled with $x, y$ are the $x, y $ components of the control, respectively.
    The green curve labeled with ``opt'' is the analytical optimal control (its $x,y$ components are the same).
    Right: the evolution of the total cost versus training steps.
    The modified costs (with running costs dropped) are plotted for the warm-up stage in the inset plot at log scale. 
}%
    \label{fig:rl_results}
\end{figure}

%% file: hmp-6-future-works.tex
\section{Summary and future work}%
\label{sec:discussions_and_future_work}
This paper presents a numerical algorithm for finding the optimal control of dynamical systems with autonomous state jumps. 
We consider the dynamical system that models the collisional rigid body dynamics, a critical component of many practical applications. 
We suggest various numerical techniques to mitigate the effects of discontinuities.
We test the proposed algorithm on several examples in which the optimal controls drive the discs to collide in order to achieve the lowest total cost.
Comparisons to HMP-MAS, the direct method, and deep reinforcement learning demonstrate the superiority of the proposed algorithm in terms of accuracy, efficiency, and flexibility.
Implemented with the forward-Euler scheme, the algorithm shows linear convergence with respect to the iteration steps and asymptotic first-order accuracy in terms of time discretization.

The proposed method can be extended in a few directions.

\paragraph*{Complex hybrid systems}
Generally, rigid bodies often possess more intricate shapes. These forms introduce additional challenges in computing the collision detection functions $\psi$, their derivatives, and the jump functions $g$. Additionally, angular factors like orientation, angular velocity, and torques are often essential for representing the system. Fortunately, these elements can be completely characterized within the framework presented in Section \ref{sec:formulation_and_assumptions}. Furthermore, even complex aspects such as state constraints or high-dimensional state spaces can be addressed by the proposed methods with proper adjustments, which will be explored in future studies.

\paragraph*{Frictional dynamics}%
Frictional force is critical in many applications, from simple models like the hopper \cite{erez2011infinite} and swimmer \cite{coulom2002reinforcement},
to complex models like quadrupeds \cite{quadrupedal}  and humanoids \cite{humanoid, humanoid2}.
The frictional forces are always present together with contacts/collisions. 
Friction introduces additional discontinuities and nonlinearities into the dynamical system, beyond those caused by collisions and contacts.
They can be viewed as a set-valued function consisting of two modes; the following formula is a one-dimensional example,
\[
    f(v) \in \begin{cases}
        \{-\mu N\} \quad &\text{if } v > 0 \\
        \{\mu N\} \quad &\text{if } v < 0 \\
        [-\mu N, \mu N] \quad &\text{if } v = 0
    \end{cases}
,\] 
where $\mu$ is the frictional coefficient, $N$ is the normal contact force, and  $v$ is the relative velocity.

Frictional dynamics can also be formulated within the framework of switching systems.
The presence of frictions adds additional dependency of the discontinuities on the state, \ie the relative velocity.
Furthermore, friction is often a source of nonlinearity, as it takes values within a cone determined by the magnitude of the normal contact force, which is itself a complex function of the state.

\paragraph*{Convergence and algorithm improvement} 
This paper employs the relaxation update scheme \eqref{eq:u_relax_update} to facilitate the convergence of the MSA. However, more advanced methods could further enhance the algorithm, including performing a line search instead of using a fixed $\alpha$, or utilizing adaptive update strategies as discussed in \cite{msa2, msa3}. These approaches might prove valuable for tackling more difficult problems characterized by significantly nonlinear dynamics and intricate discontinuous behaviors. A thorough investigation of various techniques for discontinuous optimal control problems, along with their convergence analysis, is left for future work.

\paragraph*{Optimal feedback control} 
The proposed method can act as a data generator for supervised learning algorithms, enabling the training of neural networks to learn optimal feedback controllers, similar to approaches used in conventional optimal control problems \cite{nakamura2021adaptive,zhang2022initial,zhao2024offline}. For hybrid systems, however, an additional challenge arises from the significant differences between optimal controllers across the system's various modes. As a result, it may be necessary to train multiple controllers, each tailored to a specific system mode.

%% file: hmp-7-appendix-algorithms.tex
\section{Addressing oscillation in the vicinity of collision}
\label{app:stable_estimate_near_collision}
In the experiments, we observe that the control values oscillate significantly in proximity to the collision, deteriorating the convergence of the proposed method. 
This phenomenon is attributed to the discontinuous nature of the control and the fluctuation of simulated collision indices.
To elucidate, take a collision index $c\in\cC$.
From the forward dynamics \cref{eq:fd_case1,eq:fd_case2} and the backward dynamics \cref{eq:bw_dis_jump}, we notice a discontinuous transition in both the state $\bx$ (forward) and costate $\blambda$ (backward) between $t_c$ and $t_{c+1}$.
Then, the control values $\bu_c, \bu_{c+1}$ are updated to disparate values following \cref{eq:pmp_max_u_dis,eq:u_relax_update_dis}; $\bu_{c+1}$ update to the ``after-collision'' target.
If in the next iteration, this collision index shifts (\eg from $c$ to $c'=c+1$),
$\bu_{c'}$ (which is $\bu_{c+1}$) will update to the ``before-collison'' target.
This results in $\bu_{c+1}$ oscillating between two different values, failing to converge. 
In essence, the control values in the vicinity of collisions might alternate between two distinct targets $\bu^{opt,\pm}$, where $\bu^{opt,-}, \bu^{opt,+}$ are the optimal controls before/after a collision, respectively, leading to endless oscillation between these targets.
To circumvent this oscillation, we should rely on the control values farther from the collisions to make more stable estimations since the controls are continuous in each segment.
As the control after/before a collision is right/left continuous, we will use the control values after/before the collision to approximate $\bu_{c+\frac{1}{2}}^+/\bu_{c+\frac{1}{2}}^-$.
To be specific, when estimating $\bu_{c+\frac{1}{2}}^+,\bu_{c+\frac{1}{2}}^-$,
instead of using the control values at  $c^*+1,c^{*}$ from \cref{eq:control_resolve_steepest},
we extrapolate them by the control values in $\cI_{j,p}^+=\{c^*+1+i:i\in\{j, \cdots, p\}\}$ and $\cI_{j,p}^-=\{c^*-i:i\in\{j, \cdots, p\}\}$,
respectively, where $p\ge j\ge1$ are integers. 

We remark that choosing $p$ and the extrapolation scheme depends on the desired order of accuracy and the stability of underlying dynamics, which reflects a trade-off between stability and accuracy.
If a higher-order integration scheme is applied to solve the forward and backward dynamics,
one should apply a higher-order extrapolation method with $p>j$.
Besides, one may take larger $p, j$ if the unstable dynamics brings large fluctuations in the collision times during iterations.
We note that other than the accuracy concerns, the distances (in time $t$) to adjacent collisions prohibit selecting large $p$.

%% file: hmp-7-appendix-1-exp-details.tex
\section{Details of the Two Discs Problems}
\label{app:details_of_exp}
In this section, we describe the dynamics, costate dynamics, collision detection, and jump functions of the system of two discs studied in \Cref{sec:experiments}. 

Denote the masses of these two discs by $m_1, m_2$ and the radii by $r_1,r_2$. 
In this paper, we always take $ m_1=m_2=1, r_1=r_2=0.2$ and the coefficient of restitution $C_R=1$, \ie perfectly elastic collision.
These parameters are summarized in Table \ref{tab:example_settings}.

\begin{table}[!htpb]
    \begin{center}
    \caption{Common parameters for all experiments}
    \label{tab:example_settings}
    \begin{tabular}{ccccccc}
        \hline
        \multicolumn{2}{c}{mass} & {} &\multicolumn{2}{c}{radius} & {} & coefficient of restitution \\
        \hline
        $m_1$ & $m_2$ & & $ r_1$ & $ r_2$ & & $C_R$\\
        1 & 1 & & 0.2 & 0.2& & 1\\
        \hline
    \end{tabular}
    \end{center}
\end{table}

\paragraph*{The dynamics of two discs}
The control $\bu \in \bR^2$ denotes the force applied to disc 1.
The field $\bff$ is
\begin{equation}
    \bff(\bx, \bu) = \begin{pmatrix}
        0_4 & I_4\\
        0_4 & 0_4
    \end{pmatrix} \bx 
    + \begin{pmatrix}
        0_{4\times 2}\\
        M_1^{-1}\\
        0_{2\times 2}
    \end{pmatrix}\bu,
\end{equation}
where $I_4$ is a $4\times 4$ identity matrix and $M_1=\text{diag}(m_1, m_1)$ is the inertial matrix.
$0_k = 0_{k\times k}, 0_{m\times n}$ are zero matrices.
\paragraph*{The collision function}
The contact normal vector $\cN$ is the unit vector in the collision axis
\begin{equation}
    \label{eq:eg_contact_normal}
    \cN = \frac{\bx_{q_1} - \bx_{q_2}}{\vert\bx_{q_1} - \bx_{q_2}\vert}.
\end{equation}
We denote $\bv_1=\bx_{v_1}, \bv_2=\bx_{v_2}$ for simplicity.
Let $\bv_{i, E}, i=1,2$ be the velocities after (elastic) collision,
\begin{subequations}
    \begin{align*}
        \bv_{1,E} &= \bv_1  - \frac{2m_2(\bv_1-\bv_2)\cdot \cN}{m_1+m_2}\cN, \\
        \bv_{2,E} &= \bv_2  + \frac{2m_1(\bv_1-\bv_2)\cdot \cN}{m_1+m_2}\cN.
    \end{align*}
\end{subequations}
Then the after collision velocities are,
\begin{equation}
    \bv_{i,*} =\bv_{i,E} +(C_R - 1)(\bv_{i,E}\cdot \cN) \cN - (C_R - 1)(\bv_{\text{COM}}\cdot \cN) \cN,
\end{equation}
where $i=1,2$ and $C_R\in[0, 1]$ is the coefficient of restitution and $\bv_{\text{COM}} = (m_1\bv_1+m_2\bv_2)/(m_1+m_2)$ is the center of mass velocity for the two disc system.
The positional components of $\bfg$ are kept unchanged.
In summary, the collision function $\bfg$ is,
\begin{equation}
\label{eq:inter_disc_jump}
\bfg(\bx) = [\bx_{q_1}^T, \bx_{q_2}^T, \bv_{1,*}^T, \bv_{2,*}^T]^T.
\end{equation}
\paragraph*{The collision detection}
In the problems considered here, collisions happen when the distance between the two bodies equals 0, and tends to decrease further.
The distance between the two discs can be computed by subtracting the sum of their radii from the distance between their centers.
The distance will keep decreasing if $(\bx_{q_1} - \bx_{q_2})\cdot(\bx_{v_1} - \bx_{v_2})<0$, or equivalently, $\cN\cdot (\bx_{v_1} - \bx_{v_2}) < 0 $, recalling that $\cN$ is the contact normal defined in \cref{eq:eg_contact_normal}.
Therefore, we can take
\begin{equation}
\label{eq:inter_disc_collision_detection}
\begin{split}
    \psi(\bx)
    =\frac{\sqrt{2}}{2}\times
    \begin{cases}
        \vert\bx_{q_1}-\bx_{q_2}\vert - r_1 - r_2\quad&\text{if }\cN\cdot(\bx_{v_1} - \bx_{v_2})<0,\\
        1\quad &\text{elsewhere}.
    \end{cases}
\end{split}
\end{equation}
Note that though $\psi$ is not continuous in $\bR^8$, it is smooth in the following open set of interest with some $0<\xi<r_1+r_2$
\[\{\bx\in\bR^8:(\bx_{q_1}-\bx_{q_2})\cdot(\bx_{v_1} - \bx_{v_2})<0, \vert\bx_{q_1} - \bx_{q_2}\vert>\xi\}.\]
For $\bx$ in this set, 
\[
\psi_{x}(\bx) = \frac{\sqrt{2}}{2}[\cN, -\cN, 0_2, 0_2],
\]
where $0_2\in\bR^2$ is the zero vector.
We note that the dot product 
$\sqrt{2}\psi_{x}  \bff = \cN\cdot (\bx_{v_1} - \bx_{v_2})$ is the relative velocity, whose negativity implies that $\eta$ can be solved from \cref{eq:backward_eta}.

\paragraph*{The Hamiltonian}
The Hamiltonian for the problem (except Section \ref{sub:L1-cost}) is given by
\begin{equation}
    \label{eq:eg_hamiltonian}
    H(\bx,\blambda, \bu) =  \blambda_{q_1}\cdot \bx_{v_1}+\blambda_{q_2}\cdot \bx_{v_2} + \frac{1}{m_1} \blambda_{v_1}\cdot \bu + \epsilon\vert\bu\vert^2.
\end{equation}
Thus, the direct minimizer of $H$ is
\begin{equation}
    \label{eq:eg_u_minimizer}
    \bu_* = \arg\min_{\bv \in U} H(\bx,\blambda,\bv) = -\frac{\blambda_{v_1}}{2\epsilon},
\end{equation}
and $H_{u}$ is
\begin{equation}
    \label{eq:eg_h_u}
    H_{u}(\bx, \blambda, \bu) = \blambda_{v_1} + 2\epsilon \bu.
\end{equation}

\subsection{Wall Disc Collision}
\label{app:wall_disc_colllision}
In this section, in addition to the previous collision detection function and jump function defined in \cref{eq:inter_disc_collision_detection,eq:inter_disc_jump}, we describe the functions of wall-disc collision of the example in \Cref{sub:multiple_collisions}.

The wall is represented by a hyperplane, \ie a straight line in 2D, whose distance to the origin $O$ is $b$. We denote the normal vector of the wall by $\cN_w$ such that $\cN_w\cdot OP >0$ for any point $P$ on the wall.
Without loss of generality, we assume the two discs are on the same side of the wall as the origin $O$, which means the distance from disc $i$ to the wall, $b - \cN_w\cdot\bx_{q_i}  - r_i$, is positive.
The distance will decrease if $\cN_w\cdot\bx_{v_i}  > 0$, \ie the velocity of disc $i$ projected to the wall normal is positive.
The collision detection function of the wall to disc $i$ can be defined as,
\begin{equation*}
    \psi_i(\bx) = 
    \begin{cases}
        b - \bx_{q_i}\cdot \cN_w - r_i\quad&\text{if }\cN_w\cdot \bx_{v_i}>0\\
        1\quad &\text{elsewhere}
    \end{cases}.
\end{equation*}
The corresponding jump function $\bfg_i$ will only change the velocity of disc $i$ as
\begin{equation*}
    [\bfg_i(\bx)]_{v_i} = \bx_{v_i} - 2(\bx_{v_i}\cdot \cN_w) \cN_w,
\end{equation*}
while all the other components of $\bfg_i(\bx)$ remain the same as that of $\bx$.

\subsection{Analytical optimal controls}%
\label{sub:analytical_optimal_controls}
The common parameters are summarized in Table \ref{tab:example_settings}.
In the examples considered, the optimal control should drive disc 1 to collide with disc 2 so that disc 2 is close to the origin at the terminal time $T=1$.
Recall that we always assign zero initial velocity to disc 2.
Let us assume that disc 1 collides with disc 2 with velocity $\bv'$ at the collision time $s$. 
The after collision velocity of disc 2 is $\bv=(\bv' \cdot \cN)\cN$, see Section \ref{sec:experiments}.
The displacement of disc 1 before collision is a linear function of $\cN$,
\begin{align}
    \bd &= \bq_1(s) - \bq_1(0) = \bq_2(s) - (r_1+r_2)\cN - \bq_1(0) \nonumber \\
      &= \bq_2(0) - \bq_1(0) - (r_1+r_2)\cN.
\end{align}
Here, $ \bq_1, \bq_2 $ are the positions of disc 1, disc 2, respectively.
Given the collision time $s$, the collision velocity $\bv'$ and the contact normal $\cN$, the optimal control $\bu$ before collision can be obtained by solving
\begin{mini!}|l| {u}{\int_{0}^{s}L(\bx, \bu) dt} 
    {\label{eq:opt_reformulate}}{}  
    \addConstraint{\int_{0}^{s}\bu(t)~dt=\bv' }
    \addConstraint{\int_{0}^{s}(s-t)\bu(t)~dt=\bd.}
\end{mini!}
The optimal control after collision is identically zero.
In the case $L(\bx, \bu) = \epsilon \vert\bu\vert^2$, the optimal control before collision is
\begin{equation}
    \bu(t) = \frac{6\bd}{s^2} - \frac{2\bv'}{s} + ( \frac{6\bv'}{s^2} - \frac{12\bd}{s^3}  )t.
\end{equation}
Then, the total cost is
\begin{equation}
\label{eq:reduced_obj}
    \begin{split}
         J(\bu)=\norm{\bq_2(s) + (T-s)(\bv'\cdot \cN)\cN}^2+\epsilon(\frac{12}{s^3} \vert\bd\vert^2 + \frac{4}{s} \vert\bv'\vert^2 - \frac{12}{s^2} \bv'\cdot \bd).
    \end{split}
\end{equation}
Recall that $\bd$ is a linear function of $\cN$ and $ \bq_2(s) = \bq_2(0)$. The objective \ref{eq:reduced_obj} depends only on $s, \bv'$ and $ \cN$.

\paragraph*{Example \ref{sub:centric_collision}}
In the configuration described in Section \ref{sub:centric_collision} where the concentric collision is preferred, the optimal $\cN = \frac{1}{\sqrt{2}}(1, 1)$, and the optimal $\bv$ takes the form $\bv=(v_c, v_c), v_c\in\bR$.
Then the total cost \ref{eq:reduced_obj} is quadratic in $v_c$, whose optimal value can be analytically solved as a function of $s$.
Finally, we turn the total cost into a rational function of $s$, whose minimizer in $(0, T)$ can be computed using an analytical solution
to machine accuracy.

\paragraph*{Example \ref{sub:general_collision_for_two_balls}}
In the the general situation described in Section \ref{sub:general_collision_for_two_balls}, one cannot solve the optimal contact normal vector easily. Note that the total cost \eqref{eq:reduced_obj} is quadratic in $\bv'$, we can first minimize it with respect to $\bv'$. To this end, we set the gradient of the total cost with respect to $\bv'$ to zero to have
\begin{equation}
    \bv'\cdot \cN^\perp = \frac{3\bd\cdot \cN^\perp}{2s}, \bv'\cdot \cN = -\frac{s^2(T-s)\bq_2(0) \cdot \cN - 6\epsilon \bd\cdot \cN}{4\epsilon s + s^2(T-s)^2},
\end{equation}
where $\cN^\perp=\big(\begin{smallmatrix}0 & -1\\ 1 & 0 \end{smallmatrix}\big)\cN$ is obtained by rotating $\cN$ 90 degrees.
Then, the total cost \eqref{eq:reduced_obj} is rational in  $s$ and is a fourth-order polynomial in the coordinates of $\cN$. The constraints are $\vert\cN\vert^2=1$ and $s\in(0, T)$.
Therefore, minimization of the total cost can be solved to machine accuracy using an analytical solution.

\paragraph*{Example \ref{sub:multiple_collisions}}
In the multiple collisions example described in Section \ref{sub:multiple_collisions}, 
the velocity of a disc with velocity $\bv=(\bv_x,\bv_y)$ becomes $(\bv_x, -\bv_y)$ after a collision with the wall $y=b$.
Let us denote the velocity of disc 2 after the collision with disc 1 by $\bv = (\bv_x, \bv_y)$.
It takes time  $(b-r_2-\bq_{2,y}(s)) / \bv_y$ for disc 2 to collide with the wall, where $\bq_{2,y}(s)$ is the $y$ coordinate of the position of disc 2 at time $s$.
Thus, the terminal cost is
\begin{equation*}
\begin{split}
     &\norm{\bq_{2,x}(s)+\bv_x(T-s)}^2
     + \norm{(b-r_2) - \bv_y(T-s-\frac{b-r_2-\bq_{2,y}(s)}{\bv_y})}^2 \\
    =& \norm{\bq_{2,x}(s)+\bv_x(T-s)}^2 + \norm{\bq_{2,y}(s)+(T-s)\bv_y - 2(b-r_2)}^2,
\end{split}
\end{equation*}
namely, the square distance to the point $(0, 2(b-r_2))$ as if there is no wall.
The optimal control for this case can be computed using the same procedure as in the previous examples.

%% file: hmp-7-appendix-multiple-collisions.tex
\section{Consolidation of Multiple Collision Detection Functions}
\label{app:merge_multiple_conllision}
In the system outlined in \Cref{sec:formulation_and_assumptions}, we introduced a single scalar collision detection function $\psi$ and its associated jump function $\bfg$.
However, applications with multiple collisions intuitively provide multiple collision detection functions, \eg each corresponding to a pair of potentially colliding rigid bodies. See \Cref{sub:multiple_collisions} for an example.
To elaborate, one assume a vector-valued detection function $\bpsi^* = [\psi_1, \cdots, \psi_{n_c}]^T: \bR^m\rightarrow\bR^{n_c}$, where $n_c$ is the number of possible colliding rigid body pairs.
Each element $\psi^*_i$ is dedicated to detecting collisions for a distinct pair of rigid bodies.
Correspondingly, there are also $n_c$ associated jump functions $\bfg^*_1, \cdots, \bfg^*_{n_c}$. We assume $\psi^*_i$ and $\bfg^*_i$ are smooth.
Next, we show how this vector-valued approach can be transformed into the scalar-valued formulation introduced in \Cref{sec:formulation_and_assumptions}, under the assumption of isolated collisions, \ie at most one collision occurs at any given time.

To this end, we let $\mathcal{K}_i = \{\bx: \psi_i^*(\bx) = 0\}, i=1,\ldots ,n_c$  and $B(\bx,\delta)$ be the open ball with radius $\delta$ centered at $\bx$.
The domain of interest is given by
\begin{equation*}
    \cX=\left\{\bx:
    \begin{array}{cc}
     \exists\delta>0,\,i\in\{1,\dots,n_c\} \text{ such that } \\
    B(\bx, \delta) \cap \mathcal{K}_i\ne\emptyset, B(\bx, \delta) \cap \mathcal{K}_j=\emptyset \text{ for }j \neq i
    \end{array}
    \right\}.
\end{equation*}
One can think of this set as follows. We cover the isolated collision states, the union of all jumping manifolds $\mathcal{K}_i$ subtracting their pairwise intersections, by open balls intersecting precisely one of those manifolds.
According to the definition, for any $\bx\in\cX$, we can define the unique active index $i(\bx)$ such that the only possible collision in the neighborhood of $\bx$ is on $\mathcal{K}_{i(\bx)}$.
We then define, for $\bx\in\cX$,
\begin{subequations}
    \begin{align}
    \psi(\bx) &= \psi^*_{i(\bx)}(\bx),\\
    \bfg(\bx) &= \bfg^*_{i(\bx)}(\bx).
    \end{align}
\end{subequations}
$\psi$ and $\bfg$ are smooth in $\cX$ since for every $\bx\in\cX, i(\bx)$ is constant in a ball containing $\bx$.

Above is a rigorous derivation, in numerical simulations, however, we can straightforwardly iterate over each of $\psi_i$ to find out the active index.
Actually, let us further assume the existence of a positive gap $\xi>0$ such that each time interval shorter than $\xi$ contains at most one collision.
We then can go through the prediction of collision time for each $\psi_i, i=1,\ldots,n_c$, and let the active index be the one on which the prediction is smaller than $\Delta t$, if there is one. 
There will be at most one such index if $N$, the number of time discretization, is greater than $T/\xi$ (might be multiplied by a factor if taking into account the error in estimating collision times).